%% file: PHactionsDX.tex
\documentclass[11pt]{amsart}

\usepackage[utf8x]{inputenc}

\usepackage{amsmath, amsthm, amsfonts, amssymb}
\usepackage{epsfig, enumerate}
\usepackage{color}

\usepackage{color}

\usepackage{hyperref}
\usepackage[all]{xy}

\input{makros.tex}
\def \Zk{{\mathbb Z^k}}
\def \Diff{{\rm Diff}}
\def \al{{\alpha}}
\def \Rk {{\mathbb R^k}}

\numberwithin{theorem}{section}
\numberwithin{equation}{section}

\author{Danijela Damjanovi\'c}

\address[Damjanovi\'c]{Department of mathematics, Kungliga Tekniska högskolan, Lindstedtsvägen 25, SE-100 44 Stockholm, Sweden.} 
\email{ddam@kth.se}

\author{Disheng Xu}

\address[Xu]{Department of mathematics, Kungliga Tekniska högskolan, Lindstedtsvägen 25, SE-100 44 Stockholm, Sweden.}
\email{dishengxu1989@gmail.com}


\begin{document}

\title[Partially hyperbolic abelian actions]{On conservative partially hyperbolic abelian actions with compact center foliation}


\footnote{Based on research supported by Swedish Research Council grant 2015-04644}

\begin{abstract} We consider smooth partially hyperbolic volume preserving $\ZZ^k-$actions on smooth manifolds, with uniformly compact center foliation. We show that under certain irreducibility condition on the action, bunching and uniform quasiconformality conditions, the action is a smooth fiber bundle extension of an Anosov action, or the center foliation is pathological. We obtain several corollaries of this result. For example, we prove a global dichotomy result that any smooth conservative circle extension over a maximal Cartan action is either essentially a product of an action by rotations and a linear Anosov action on the torus, or has a pathological center foliation.  
\end{abstract}

\maketitle

\maketitle
\section{introduction}

In this paper we make a step towards understanding smooth  volume preserving partially hyperbolic abelian actions. We consider $\ZZ^k$ actions on smooth manifolds  $M$ such that the action has sufficiently many partially hyperbolic volume preserving diffeomorphisms of $M$, each of which is hyperbolic transversally  to a common center  foliation. If there is no center foliation, the action is Anosov. Smooth classification of Anosov actions of abelian groups  of rank greater than 2 has been a topic of interest for few decades already and many local and global classification results exist, \cite{KS},  \cite{KS06},  \cite{HW},  \cite{FKS}, \cite{KalSpa}, to name a few. Conjecturally, all Anosov abelian  actions which do not reduce to actions generated by a single Anosov diffeomorphism, are smoothly conjugate to the known algebraic examples \cite{KS}.

When it comes to partially hyperbolic abelian actions of higher-rank abelian groups, the existing smooth classification results are all local, in a neighborhood of algebraic examples, for example \cite{DK-KAM},  \cite{DF}, \cite{DK-PH2}, \cite{VW}. It is not expected that any global classification in general is possible. 
In this paper we restrict ourselves to the subclass of partially hyperbolic actions whose center foliation has compact leaves. Basic examples of such actions are product actions of Anosov actions and isometric actions on compact manifolds. More general examples are skew products of this type, or more generally fibre bundles and fibrations.  Our goal here is to show that under certain higher rank and certain geometric assumptions, partially hyperbolic actions with compact centre leaves reduce to smooth fibrations. Moreover, these fibrations are \emph{essentially} (up to a cover or a restriction on a finite index subgroup of $\ZZ^k$) products over \emph{algebraic} Anosov actions. 

There are however examples of partially hyperbolic diffeomorphisms where all center leaves are compact, but the projection map to the manifold quotient by center foliation is not even a fibration \cite{BW}. Therefore we work here with a stronger assumption, which is already established as the most reasonable one, that the center foliation of the partially hyperbolic action is \emph{uniformly compact} (or with \emph{trivial holonomy}) which means that the center leaves have finite (resp. trivial) holonomy. We remark that it has been conjectured that compactness of the center foliation implies uniform compactness (or trivial holonomy up to a finite cover) and that in some special situations this has been confirmed \cite{Gogolev} and \cite{Carrasco}. One important feature implied by uniform compactness of the center foliation is \emph{dynamical coherence} \cite{BB}. 

In rank-one situation, that is for a single partially hyperbolic diffeomorphism with uniformly compact center foliation,  under certain assumptions\footnote{For examples,  uniformly quasiconformality assumptions on stable and unstable bundles, small dimensions of the manifold, accessibility, certain bunching conditions, non-pathological center foliation, etc.} it is proved in \cite{AVW2}, \cite{BX} that smooth partially hyperbolic volume preserving diffeomorphisms are smooth fibrations over an Anosov diffeomorphism. Here, we first generalize these rank-one results to $\ZZ^k$ partially hyperbolic actions (Theorem \ref{main} in Section \ref{Fibrations}), under certain irreducibility assumption on the action, while neither element of the action may satisfy the conditions of the rank-one results in \cite{AVW2}, \cite{BX}. Furthermore, under extra higher-rank assumptions, we show that the action is essentially a product of an Anosov action which is smoothly conjugated to an algebraic Anosov action on a torus, and a $\ZZ^k-$action on a compact manifold (or even an action by translations of a compact Lie group), see Theorems \ref{main: global rigidity}, \ref{main: gl r dim 1} in section \ref{subsection: main results higher ranks}. One corollary of our result is the following global dichotomy result for certain partially hyperbolic actions (for precise definition of concepts involved  we refer to section \ref{subsection: restate thm mainintro}).
\color{black}

\begin{maintheorem}\label{main: intro}Any $C^\infty$ conservative circle extension over a maximal Cartan action is either essentially a product of an action by rotations and a linear Anosov action on a torus,  or has a pathological center foliation.
\end{maintheorem}

Notice that similar "\textit{pathological center foliation or rigidity}" dichotonomy results are proved in \cite{AVW}, \cite{AVW2} for certain class of $3-$dimensional single partially dynamical diffeomorphisms, Theorem \ref{main: intro}, and further results in our paper could be viewed as the higher-rank and higher-dimensional extension of results in \cite{AVW}, \cite{AVW2}. Rather than using accessibility property of  partially hyperbolic diffeomorphisms in \cite{AVW} and  invariance principle developed in \cite{AV}, \cite{AVW2}, our work here improves some rank one arguments from \cite{BX} and uses higher rank results (essential cocycle rigidity in \cite{DX} and \cite{NTnonabelian}, and a  conservative extension of \cite{KS07}).\color{black}


So far global rigidity results for higher rank abelian actions only exist for \emph{Anosov} actions. For certain classes of Cartan actions, global classification result is obtained in \cite{KalSpa}. For \emph{uniformly quasiconformal} higher rank Anosov actions global rigidity results were obtained in \cite{KS06} and \cite{KS07}. In our case, the actions we obtain in the base of fibration are not necessarily uniformly quasiconformal, so we do not use results of \cite{KS06} and \cite{KS07}, but we prove a conservative version of global rigidity result for Anosov actions in this case to show existence of a smooth conjugacy in our set-up (Theorem \ref{main: K-S conservative} in Section \ref{GlobalAnosov}). Another by-product are new global rigidity results for single diffeomorphisms (Theorem \ref{main for single diffeo} and Theorem \ref{main for single diffeo'} in Section \ref{RankOne}). 

We note here that the irreducibility condition we use for partially hyperbolic actions in this paper, is defined in terms of the Weyl chamber picture in the acting group, and it mimics to some extent the Weyl chamber picture for maximal Cartan actions. In particular, the irreducibility condition we use implies the \emph{totally non-symplectic} and \emph{resonance-free} conditions used  in  \cite{KS06} and \cite{KS07}. (Various irreducibility conditions for abelian actions are defined in Section \ref{subsection: def PH}).

\section{Setting and statements}

\subsection{Definitions of partially hyperbolic actions and $PH_{\mathrm{vol}}^\infty(k, M)$} \label{subsection: def PH}
An Anosov abelian action is a smooth $\ZZ^k-$action on a smooth manifold which contains an Anosov element.

In this paper we will study  \emph{partially hyperbolic (PH)} $\ZZ^k-$actions on  smooth manifolds. Recall that a $C^1-$diffeomorphism $f$ on a compact smooth manifold $M$ is called partially hyperbolic if there is a $Df-$invariant splitting $E^s\oplus E^c\oplus E^u $ of $TM$ such that there exists $k>1$ such that for any $x\in M$ and any choice of unit vectors $v^s\in E^s$, $v^c\in E^c$, $v^u\in E^u$, $$\|Df^k(v^s)\|<1<\|Df^k(v^u)\|$$
$$ \|Df^k(v^s)\|<\|Df^k(v^c)\|<\|Df^k(v^u)\|$$
Smooth action $\al:\ZZ^k\to \Diff^\infty(M)$ is called partially hyperbolic if there exists $a\in \ZZ^k$ such that $\al(a)$ is a partially hyperbolic diffeomorphism. By commutativity all the elements of $\al$ preserve $E^s_{a}\oplus E^u_{a}$ and $E^c_{a}$. Therefore without loss of generality we make the following \textbf{standing definition} for this paper:  
\begin{definition}\label{def: ph splt}An action $\al:\ZZ^k\to \Diff^\infty(M)$ is called \emph{partially hyperbolic} if there exist a $D\al-$invariant continuous splitting of $TM=E^c(\al)\oplus E_T(\al)$ and a \emph{partially hyperbolic element} $a\in \ZZ^k$ in the sense that $\al(a)$ is a partially hyperbolic diffeomorphism and associated $D\al(a)-$invariant splitting satisfies
\begin{equation}\label{eqn: def ET}
E_T(\al)=E^u_{a}\oplus E^s_{a}, E^c=E^c_{a}
\end{equation}
The splitting $E_T(\al)\oplus E^c(\al)$ is called \emph{partially hyperbolic splitting associated to $\al$}. Any elment $a\in \ZZ^k$ such that $\al(a)$ is a partially hyperbolic diffeomorphism and satisfies \eqref{eqn: def ET} is called a \emph{partially hyperbolic element (of $\al$}).


\end{definition}


In this paper we will study \textit{conservative} partially hyperbolic action, i.e. we always assume $\al$ is preserving an ergodic volume $\nu_M$ on $M$. Then there are finitely many linear functionals $\chi$ on $A$, called  {\em Lyapunov functionals}, a set of full measure $\Lambda$ and  a measurable splitting of the bundle $T_{\Lambda}(E_T)=\bigoplus \limits_{\chi}E^{\chi}$, 
such that for $v\in E^{\chi}$ and $b\in A$ the Lyapunov exponent of $v$ with respect to $\al(b)$ is $\chi(b)$.  Lyapunov functionals $\chi$ can easily be extended from $\ZZ^k$ to $\RR^k$. If $\chi$ is a non-zero Lyapunov exponent then we define its \emph{coarse Lyapunov subspace} by $$E_\chi:=\bigoplus\limits_{\{\lambda=c \chi:c>0\}} E^\lambda.$$
The connected components of the set $\bar A=\bigcap\limits_{\chi\neq 0}(\ker\chi)^c$ are called Weyl chambers for action $\alpha$.  See \cite[Section 5.2]{KalKat} for more details. The elements in $a\in \bar A$ are called \emph{regular}. If $E^c$ is trivial, $\alpha$ is Anosov.  

In the study of Anosov actions it is very natural to assume the presence of sufficiently many Anosov elements. For example in the proof of global classification of Anosov action on infranilmanifold \cite{FKS}, the authors assume that there exists an Anosov element in each Weyl chamber. In this paper we assume similarly that in any Weyl chamber $\mathcal{C}$ there is a partially hyperbolic element $b_\mathcal{C}\in \mathcal{C}$.
\begin{definition}\label{def: PHmathrmvolinfty(k, M)}We denote by $PH_{\mathrm{vol}}^\infty(k, M)$ the space of  conservative partially hyperbolic $\ZZ^k$ actions $\al: \ZZ^k\to \Diff^\infty_{\mathrm{vol}}(M)$ satifying that for any Weyl chamber there exists a partially hyperbolic element of $\al$ in it.
\end{definition}



We denote by $\{\chi_i,i=1,\dots\}$  a maximal collection of non-zero Lyapunov exponents that are not positive multiples of one another. 
The action $\al$ is called \emph{full}\footnote{In \cite{DX} we also define similar fullness assumption for Anosov actions which is basically the fullness condition here combined with a higher rank assumption.} if for every coarse Lyapunov distribution $E_\chi$, there exist a Weyl chamber $\mathcal{C}$ and a partially hyperbolic element of $\al$, $b_\mathcal{C}\in \mathcal{C}$ such that $$E^s_{b_{\mathcal{C}}}= \bigoplus_{\lambda\neq c\chi,c>0} E_\lambda,\quad  \nu_M-a.e.$$ and $$E^u_{b_{\mathcal{C}}}=E_\chi, ~~~\nu_M-a.e.$$
Under fullness assumption all coarse Lypunov distributions are not only measurable, but also globally defined H\"older continuous distributions which integrate to H\"older  foliations with smooth leaves. We denote the foliation associated to $E_\chi$ by $W^\chi$. Similar to the case of Anosov actions under the fullness assumption we have $$E_\chi=\bigcap_{b_{\mathcal{C}}, \chi(b_{\mathcal{C}})>0} E^u_{b_{\mathcal{C}}},\quad W^\chi:=\bigcap_{b_{\mathcal{C}}, \chi(b_{\mathcal{C}})>0} W^u_{b_{\mathcal{C}}}$$
where $W^u_{b_{\mathcal{C}}}$ is the unstable manifold for the partially hyperbolic diffeomorphism $\al(b_\mathcal{C})$. $W^\chi$ is called the \emph{coarse Lyapunov foliation} associated to  $E_\chi$.

For $\al\in PH_{\mathrm{vol}}^\infty(k, M)$ we will also use the following notions which all describe to what extent is $\al$ irreducible:

- $\al$ is \emph{maximal} if $k\geq 2$ and there are exactly $k+1$ coarse Lyapunov exponents which correspond to $k+1$ distinct Lyapunov hyperplanes, and if Lyapunov hyperspaces are in general position (namely, if no Lyapunov hyperspace contains a non-trivial intersection of two other Lyapunov hyperspaces). 
 
-  $\al$ is \emph{totally non-symlpectic} (TNS) if there are no negatively proportional Lyapunov exponents. 

- $\al$ is \emph{Cartan} if all coarse Lyapunov distributions are one-dimensional.

- $\al$ is \emph{resonance-free} for any Lyapunov functionals $\chi_i$ , $\chi_j$, the functional $(\chi_i − \chi_j )$ is not proportional to $\chi_l$.

We say a $\ZZ^k-$action is \emph{rank-one} if all the Lyapunov hyperplanes coincide and there are only two Weyl chambers.  
\subsection{Uniformly compact center foliation, pathological center foliation and precise statement of Theorem \ref{main: intro}}\label{subsection: restate thm mainintro}In general, the center distribution $E^c$ may not be integrable, and we are not aware of any classification results in such a general setting. In all the results below we will make extra assumptions: we will consider actions whose center distribution integrates to foliation $W^c$ with \emph{uniformly compact} leaves, for precise definition see section \ref{subsection: trivial holonomy for foliation}. Notice that this may not however even imply that the space of leaves is even Hausdorff. In most of the following cases we will assume in addition that all center leaves have \emph{trivial holonomy} which covers many cases of interests since this condition is equivalent to the condition that $M$ is a topological fibration over $M/W^c$ and $M/W^c$ is a topological manifold. For the definiton of holonomy of foliation see Section \ref{subsection: trivial holonomy for foliation}.

We will say that the central foliation $W^c$ is \emph{pathological} for action $\al$ if there exists a coarse Lyapunov distribution $E^i$ such that $\dim(E^i)=1$ and $E^i\oplus E^c$ are continuous bundles which integrates to H\"older foliation $W^{i,c}$, and either $\nu_M$ does not \textit{have Lebesgue desintegration along $W^c$} or $W^c$ is not \textit{strongly absolutely continuous} in $W^{i,c}$. Basically these two conditions are related to absolute continuity for $W^c$ foliation, for precise definitions see section \ref{subsection: abs cont}.

Suppose $\al\in PH_{\mathrm{vol}}^\infty(k, M)$ is a maximal Cartan partially hyperbolic action defined in the previous section, and the common center foliation $W^c$ for $\al$ is uniformly compact and one-dimensional. In this case we call $\al$ is a \emph{$C^\infty$ conservative circle extension over a maximal Cartan action}, which is mentioned in Theorem \ref{main: intro}. Then  Theorem \ref{main: intro} can be restated as the following: suppose $\al$ is defined above, then we have
\begin{prop}[Precise statement of  Theorem \ref{main: intro}]\label{prop: restate main intro}
If $\al\in PH_{\mathrm{vol}}^\infty(k, M)$ is a $C^\infty$ conservative circle extension over a maximal Cartan action, then exactly one of the following holds:
\begin{enumerate}\item $W^c$ is pathological.
\item There is a finite index subgroup $A\subset \ZZ^k$ and a finite cover $M^\ast$ of $M$ such that the lift $\al^\ast|_{A}$ of $\al_{A}$ on $M^\ast$ is  smoothly conjugated to a product of a rotation on the circle and a linear Anosov action on a torus.
\end{enumerate} 
\end{prop}


\subsection{Basic notions, main assumptions and examples}\label{subsection: def ABC} In fact Theorem \ref{main: intro} and Proposition \ref{prop: restate main intro} are corollaries of our more general global reduction results for a class of partially hyperbolic actions. We first define few notions which will be important in narrowing down our class of partially hyperbolic actions. 

A $C^1-$diffeomorphism $f:M\to M$ is \emph{uniformly quasiconformal (UQC)} on an $f-$invariant distribution $E\subset TM$ if for for $x\in M, n\in \ZZ$
$$K^E(x,n)=\frac{\sup\{\|Df^n(v)\|:v\in E, \|v\|=1\}}{\inf\{\|Df^n(v)\|:v\in E, \|v\|=1\}}$$
is uniformly bounded in $x,n$. The quantity $K^E$ measure the failure of iterates of $Df$ to be conformal on the bundle $E$. Notice that $f$ is automatically uniformly quasiconformal if $\dim E=1$.

We will also use the classical bunching conditions, described in detail in Section \ref{subsection: bunching cond}, which describe domination of  hyperbolicity of a partially hyperbolic diffeomorphism $f$ over non-conformality in the center direction. For example $r-$bunching and $\infty -$bunching both hold if there is a continuous Riemannian metric on $E^c$ with respect to which $Df|_{E^c}$  is an isometry.

For $\al\in PH_{\mathrm{vol}}^\infty(k, M)$ we formulate and label the following conditions:

(A)  For every coarse Lyapunov distribution $E^i$, there exists $a\in \ZZ^k$ such that $\al(a)$ is uniformly quasiconformal on $E^i$ and $E^i=E^u_a$.  

(B) $\alpha$ is $r-$bunched for some $r\geq 1$, or $\dim(W^c)=1$.

(B')  $\alpha$ is $\infty-$bunched or $\dim(W^c)=1$.

(C) The center distibution $E^c$ is integrable to a center foliation $W^c$, all center leaves are compact and have trivial holonomy.

\begin{remark}
Condition (A) contains two assumptions: one is fullness of $\al$, and the other is the UQC assumption. We will show in Lemma \ref{lemma: assp A+higher rank imply tns} that fullness  implies that the action satisfies (unless it is rank-one) previously used higher-rank conditions in \cite{KS07} and \cite{KS06} such as TNS condition and resonance free condition. 
\end{remark}

\begin{remark} If $\al\in PH_{\mathrm{vol}}^\infty(k, M)$ satisfies  (A) and (C), then for any Weyl Chamber $\mathcal C$, the element $\al(b_\mathcal{C})$ defined in last section, is \emph{dynamically coherent} (see section \ref{subsection: uniform compa cent} and \cite{BB}), in particular this means that for each $i$, there exist foliations $W^{i,c}$ with smooth leaves, tangent to the distributions $E^i\oplus E^c$.  Moreover, the foliations $W^i$ and $W^c$ subfoliate $W^{i,c}$. It is conjectured that every compact center foliation is uniformly compact and any partially hyperbolic system with uniformly compact center foliation is finitely covered by one which all center leaves have trivial holonomy, cf. \cite{BB, Carrasco, Gogolev}. \end{remark}

Algebraic examples of actions which satisfy conditions (A), (B) and (C) can be obtained by taking a Cartan action on any torus and extending it by a constant principal $G-$cocycle to the product of torus and $G$, where $G$ is a compact connected Lie group. For a concrete example take an action $\bar\al: \ZZ^2\to  \Diff(\mathbb T^3)$ (Example 2.2.15 in \cite{KN}) generated by the following commuting pair of matrices:  \[
A =
 \left( \begin{array}{ccc}
     0 & 1 & 0 \\
     0 & 0 & 1 \\
     1 & 8 & 2 \\
        \end{array} \right),
B = 
 \left( \begin{array}{ccccc}
     2 & 1 & 0 \\
     0 & 2 & 1 \\
     1 & 8 & 4 \\
   \end{array} \right)
 \]
  and define  $\al: \ZZ^2\to  \Diff( \mathbb T^3\times \mathbb T^d)$, $d\ge 1$ to be $\al(a)( x,y)= (\bar\al(a)(x), y)$.  For example, on $ \mathbb T^d$,  $\ZZ^2$ can act by toral translations. Small perturbations of a class of such actions have been classified recently in \cite{DF}.
  
  In the concrete example above the coarse Lyapunov distributions of the action generated by $A$ and $B$ are 1-dimensional. In the arguments for main results of this paper we often use different strategy for the situation when coarse Lyapunov distributions are higher dimensional, and obtain stronger results in this case. Examples of actions with 2-dimensional coarse Lyapunov distributions on which each element of the action is UQC can be obtained on $\mathbb T^6\times \mathbb T^d$ as follows. On $\mathbb T^6$ let $A\in SL(6, \mathbb Z)$ have characteristic polynomial $x^6+5x^2+6x+1$. Then $A$ has 3 pairs of complex conjugate eigenvalues, 2 pairs outside the closed unit disc, and one pair inside the open unit disc. The dimension of the centralizer $C(A)$ of $A$ in $SL(6, \mathbb Z)$ is 2 and any element $B$ of $C(A)$ also has three pairs of complex conjugate eigenvalues and has same eigenspaces as $A$. These eigen spaces are coarse Lyapunov distributions for the $\mathbb Z^2$ action $\al$ generated by $A$ and $B$, and on each of them any element of $\al$ is UQC. Action $\al$ can then easily be extended to a partially hyperbolic action from $\mathbb T^6$ to $\mathbb T^6\times \mathbb T^d$ by taking identity action on $\mathbb T^d $ or by action by toral translations.

\subsection{Principal bundle extensions and essential algebraicity}\label{subsection main examp} We list here classes of algebraic examples: principal $G-$cocycles and principal $G-$extensions which fall into the set-up described above. For details cf. \cite{NTnonabelian}

Let $M$ be a compact manifold, and $G$ be a compact connected Lie group. A smooth principal $G-$bundle $P$ over $M$ is a smooth fiber bundle $P$ over $M$ such that the $G$ acts freely, transitively, smoothly and preserving the fibers. In other words, each fiber of $P$ is a $G-$homogeneous space. The \textit{trivial} principal $G-$bundle is $M\times G$ with $G$ acting on $M\times G$ from the right.

A smooth $G-$map $F: P_1\to P_2$ between principal $G-$bundles $P_1, P_2$ is a smooth map such that $F(\zeta g)=F(\zeta)g$ for all $(\zeta, g)\in P_1\times G$. Since $G$ is chosen to be compact we can fix a family of $G-$invariant Riemannian metric on the fibers such that the restriction of $G-$map to any fiber is an isometry.

Consider an action of a group $A\to \mathrm{Diff}^\infty(M)$ (for example $A:=\ZZ^k$) on the manifold $M$, $\bar\al:A\times M\to M$. A principal \textit{$G-$extension} of $\bar\al$ is a principal $G-$bundle $\pi: P\to M$ endowed with a lift of $\bar\al$ to an action $\al:A\times P\to P$ by $G-$maps. In particular if $P$ is a trivial bundle $M\times G$, an extension $\al$ over $\bar\al$ is described by a \textit{cocycle}: $$\al:A\times M\times G \to M\times G, \al(a)\cdot (x,g)=(\bar\al(a)\cdot x, \beta(a,x)g), x\in M, g\in G, a\in A$$where $\beta$ satisfies $$\beta(ab,x)=\beta(a, 
\al(b)\cdot x)\beta(b,x)$$ If $\beta$ does not depend on $x\in M$ then $\beta:A\to G$ is a group homomorphism and $\beta$ is called a constant $G-$cocycle over $\bar\al$. In this case we call $\al$ a product action over $\bar\al$.

A smooth $\ZZ^k-$action $\al:\ZZ^k\to \mathrm{Diff}^\infty(M)$ is called \textit{essentially algebraic} if there is a finite index  subgroup $A$ of $\ZZ^k$ and a finite cover $\pi:\hat{M}\to M$ such that the lifted action $\hat{\al}:A\to \mathrm{Diff}^\infty(\hat{M})$ is smoothly conjugated to a product  of a linear Anosov action on a torus, and an action by left multiplication on  $G$, where $G$ is a compact connected Lie group.

\subsection{Non-algebraic examples and essential products over algebraic Anosov actions}\label{subsection suspension}In this subsection we define a main class of non-algebraic example of smooth partially hyperbolic $\ZZ^k-$action: cocycle taking values in the group of diffeomorphisms, which can be viewed as a non-algebraic generalization of examples in section \ref{subsection main examp}. As before we suppose $\bar\al$ is a smooth action on a smooth manifold $M$ of a group $A$, i.e. $\bar\al: A\times M\to M$. In addition we assume $N$ is a compact smooth manifold. Similar to section \ref{subsection main examp}, we consider an extension $\al: A\times M\times N\to M\times N$ of $\bar\al$ which is given by a \textit{cocycle}
$$\al:A\times M\times N \to M\times N, \al(a)\cdot (x,y)=(\bar\al(a)\cdot x, \beta(a,x)y), x\in M, y\in N, a\in A$$here $\beta: A\times N\to \Diff^\infty(N)$ satisfies $$\beta(ab,x)=\beta(a, 
\al(b)\cdot x)\beta(b,x)$$ In the case that $\beta$ does not depend on $x\in M$ we call $\beta$ a \textit{constant cocycle} (taking values in $\Diff^\infty(N)$) over $\bar\al$. If $\beta$ is a constant cocycle, then $\al$ is a product action of    $\bar\al$ and a $\ZZ^k-$action by diffeomorphisms  of $N$.

A smooth $\ZZ^k-$action $\al:\ZZ^k\to \mathrm{Diff}^\infty(M)$ is called \textit{essentially a product over Anosov linear action} if there is a cover (\textbf{not} necessarily finite) $\pi:\hat{M}\to M$ such that the lifted action $\hat{\al}:\ZZ^k\to \mathrm{Diff}^\infty(\hat{M})$ is smoothly conjugated to a product of a smooth  $\ZZ^k$ action on $N$ and an action which is lifted from a linear Anosov action on a torus.

\subsection{Fiber bundle extensions} In this section we introduce a wider class of partially hyperbolic $\ZZ^k-$actions: \textit{fiber bundle extension} over $\ZZ^k-$actions, which contains all examples in section \ref{subsection main examp} and \ref{subsection suspension}.

\begin{definition}\label{fiber bundle extension}Suppose $M,N$ are two compact manifolds and $\pi:M\to N$ is a submersion. Assume that there exist $\ZZ^k-$actions $\al, \bar{\al}$ on $M,N$ respectively.  We say $\al$ is a \textit{$C^r-$fiber bundle extension over $\bar\al$} if  $M,N,\pi, \al,\bar\al$ are $C^r$ and $\bar{\al}\circ \pi=\pi\circ\al$. 
In addition if $\al$ is a $C^\infty$ fiber bundle extension over $\bar{\al}$ then $\al$ is called an \textit{isometric extension} over $\bar\al$ if there is a smoothly varying family of Riemannian metrics $\{d_x\}_{x\in N}$ on fibers $\{\pi^{-1}(x)\}$ such that for any $a\in \ZZ^k$, $\al(a) $ is an isometry on each fiber.
\end{definition}

\begin{remark}
Note that by classical Ehresmann's fibration theorem (cf. \cite{Ehres fibration}) if $\pi:M\to N$ is a  (smooth) surjective submersion for which $\pi^{-1}\{x\}$ are compact and connected for all $x\in N$, then $\pi$ is a locally trivial fibration, so it admits a compatible fiber bundle structure. Foliations associated to submersions are \emph{simple foliations}. For simple foliations every leaf has trivial holonomy. Conversely, if one has a foliation on $M$ such that each leaf has trivial holonomy, the leaf space is Hausdorff and each leaf has finitely generated fundamental group, then the foliation is simple, i.e. it is defined via a submersion from $M$ to the space of leaves. This justifies using the term  \textit{fiber bundle extensions} in the set-up described above.
\end{remark}

\subsection{Reduction to fiber bundle extensions over Anosov actions}\label{Fibrations}The first step to prove global rigidity results for partially hyperbolic actions is the following theorem about reduction to fiber bundle extensions.
\begin{maintheorem}\label{main}
If $\al\in PH_{\mathrm{vol}}^\infty(k, M)$ and  satisfies the conditions (A), (B) and (C), then one of the following two cases holds:
\begin{enumerate}
\item 
There is $i$ such that $\dim(E^i)=1$ and  $W^c$ is pathological. 
\item $\al$ is a $C^r$ fiber bundle extension over $\bar\al$, where $\bar\al$ is an {Anosov} action induced by $\al$ on the quotient space $M/W^c$, and {  $\bar\al$ also satisfies (A)}.  Moreover:

\begin{enumerate}
\item If $\alpha$ satisfies (B'), then $\al$ is a $C^\infty$ fiber bundle extension over $\bar\al$. 

\item If $\alpha$ satisfies (B') and $\dim(E^i)>1$ for all $i$, then  all coarse Lyapunov foliations of $\bar{\al}$ are smooth.

\item If $\dim(E^c)=1$ then $\al$ is an isometric extension over $\bar\al$.
\end{enumerate}
\end{enumerate}
\end{maintheorem}

\begin{remark}The situation (1) in Theorem \ref{main} could happen. For example, there exists single diffeomorphism with pathological center foliation satisfies all our assumptions of Theorem \ref{main} (cf. \cite{Shub-Wilkinson}). It is a very interesting question whether there are such  examples for higher rank partially hyperbolic actions.
\end{remark}


\subsection{Rigidity results for partially hyperbolic higher rank actions}\label{subsection: main results higher ranks}
Under the higher rank assumption, we can get  stronger rigidity results for partially hyperbolic abelian actions. For the subsequent two results we assume  $\al\in PH_{\mathrm{vol}}^\infty(k, M)$ is not rank-one, all  $a\in \ZZ^k-\{0\}$ are partially hyperbolic elements, and $\al$ satisfies conditions  (A), (B') and (C). Moreover we assume the case (1) in Theorem \ref{main} does \textbf{not} hold. Then we have: 




\begin{maintheorem}\label{main: global rigidity} 
$\al$ is essentially a product over Anosov linear action in the sense of section \ref{subsection suspension}.
\end{maintheorem}

\begin{maintheorem}\label{main: gl r dim 1} If  $\dim E^c=1$ then $\al$ is essentially algebraic in the sense of section \ref{subsection main examp}.

\end{maintheorem}



\subsection{Global rigidity for certain conservative Anosov $\ZZ^k-$actions }\label{GlobalAnosov}An intermediate step to prove Theorem \ref{main: global rigidity} is the following conservative extension of a global rigidity result of Anosov $\ZZ^k-$actions by Kalinin-Sadovskaya \cite{KS07}. 

We call an action $\al\in PH_{\mathrm{vol}}^\infty(k, M)$ Anosov if it has trivial center distribution. An affine automorphism  of a torus is the composite of a translation and a linear automorphism on the torus. 
\begin{maintheorem}\label{main: K-S conservative}
Let $\al\in PH_{\mathrm{vol}}^\infty(k, M)$ be an Anosov action with all non-trivial elements Anosov, which is not rank-one and satisfies (A). 
Then a finite cover of $\alpha$ is smoothly conjugated to a $\ZZ^k-$action formed by affine automorphisms of a torus.
\end{maintheorem}

\begin{remark} In \cite{KS07} the authors assumed that each element of the action is uniformly quasiconformal on each Lyapunov direction $E^i$. In the Theorem above we relax this condition to condition (A). We point out that  there are higher rank abelian actions $\al$ such that for each coarse Lyapunov distribution $E^i$, there exists $a_i$ such that $\al(a_i)$ uniformly quasiconformal on $E^i$ but most of the elements are not uniformly quasiconformal on $E^i$. For example consider the action $\al:\ZZ^2\to \Diff(\TT^8)$ generated by the matrices:
$$f=\begin{pmatrix}
C& 1\\
&C\\
&&D\\
&&&D
\end{pmatrix},
g=\begin{pmatrix}
C^2\\
&C^2\\
&&D^3&1\\
&&&D^3
\end{pmatrix}$$
where $C=\begin{pmatrix}
2&3\\1&2
\end{pmatrix}, D=\begin{pmatrix}
3&4\\2&3
\end{pmatrix}$. Then $T\TT^8=E^{1}\oplus E^{2}\oplus E^{3}\oplus E^4$, where $E^1,E^2,E^3,E^4$ are two dimensional (generalized) eigenspaces of $Df$ corresponding to  eigenvalues $2+\sqrt3, 2-\sqrt3, 3+2\sqrt2,3-2\sqrt2$ respectively. $g$ is uniformly quasiconformal on $E^1, E^2$ but not on $E^3, E^4$, $f$ is uniformly quasiconformal on $E^3, E^4$ but not on $E^1, E^2$. Notice that $$E^u_f=E^1\oplus E^3, E^u_{f^5g^{-2}}=E^1\oplus E^4$$ hence $\al$ is not a rank-one action.
\end{remark}

\begin{remark}
In Theorem \ref{main: K-S conservative}, if there are some $E^i$ such that $\dim E^i=1$ the condition that all non-trivial elelments of the action are Anosov is not necessary. Moreover in this case we can conclude $\al$ itself (not only its finite cover) is smoothly conjugated to a $\ZZ^k-$action formed by automorphisms of a torus.
\end{remark}

\subsection{Rank one results} \label{RankOne}
The result in Theorem \ref{main} is new even for the case of single diffeomorphism. Suppose $f:M\to M$ is a smooth volume preserving partially hyperbolic diffeomorphism with compact center foliation, then our arguments to prove Theorem \ref{main} imply the following theorems:
\begin{maintheorem}\label{main for single diffeo}
If $f$ is $r-$bunched for some $r\geq 1$ and $\dim E^u=\dim E^s=1$, then we have the following dichotomy:

i). $W^c(f)$ is pathological in the sense that one of the following holds:
\begin{enumerate}
\item $\nu_M$ does not have Lebesgue disintegration along $W^c$.
\item $W^c$ is not strongly absolutely continuous in $W^{cs}$.
\item $W^c$ is not strongly absolutely continuous in $W^{cu}$.
\end{enumerate}  

ii). $W^c(f)$ is a $C^r-$foliation. Moreover up to a double cover $f$ is $C^r-$fiber bundle extension over an Anosov diffeomorphism on a torus. If in addition $r=\infty$ or $\dim W^c=1$, then up to a double cover $f$ is a $C^\infty-$fiber bundle extension (or an isometric extension if $\dim E^c=1$) over an Anosov diffeomorphism on a torus.
\end{maintheorem}

\begin{maintheorem}\label{main for single diffeo'}
Suppose $f$ is $r-$bunched for some $r\geq 1$, $\dim E^u=1, \dim E^s>1$ and $f$ is uniformly quasiconformal on $E^s$. Moreover we assume $W^c(f)$ has uniformly compact leaves then we have the following dichotomy:

i). $W^c(f)$ is pathological in the sense that one of the following holds:
\begin{enumerate}
\item $\nu_M$ does not have Lebesgue disintegration along $W^c$.
\item $W^c$ is not strongly absolutely continuous in $W^{cu}$.
\end{enumerate}  

ii). $W^c(f)$ is a $C^r$ foliation. Moreover up to a double cover $f$ is $C^r-$fiber bundle extension over an Anosov diffeomorphism. If in addition $r=\infty$ or $\dim W^c=1$, then $f$ is a $C^\infty$ fiber bundle extension (or an isometric extension if $\dim E^c=1$) over an Anosov diffeomorphism.
\end{maintheorem}
When $\dim W^c=1$, a result similar to Theorem \ref{main for single diffeo} is proved in an unpublished note of Avila, Viana and Wilkinson \cite{AVW2}. For the case $\dim E^u,\dim E^s>1$ and $f$ is uniformly quasiconformal on $E^u$ and $E^s$, similar result is proved in \cite{BX} (in that case $W^c(f)$ cannot be pathological).

\subsection{Organization of the paper}
In Chapter \ref{section: basic notion} we define the regularity, absolute continuity and holonomy of foliations and give the bunching conditions for partially hyperbolic actions.  In Chapter \ref{section: preliminar} we state some useful facts from \cite{BB}, \cite{BX} and prove several useful results for absolutely continuity and regularity for dynamically defined foliations under suitable UQC and bunching assumptions. In Chapter \ref{section: proof of thm 1, dim>1} we prove Theorem \ref{main} when all $E^i$ are higher dimensional by applying the results in Chapter \ref{section: preliminar}. For the case that there exists $i$ such that $\dim E^i=1$, Theorem \ref{main} is proved in Chapter \ref{section: thm main dim ei=1} by considering the one dimensional version of arguments in \cite{BX}. In Chapter \ref{section: conservative KS} we give the proof of Theorem \ref{main: K-S conservative}. The last two chapters contain the proofs of Theorems \ref{main: intro}, \ref{main: global rigidity}, \ref{main: gl r dim 1}, \ref{main for single diffeo}, \ref{main for single diffeo'}.

\section{Basic notions and definitions}\label{section: basic notion}

\subsection{Trivial holonomy for foliation}\label{subsection: trivial holonomy for foliation}   The goal of this subsection is to define \textit{holonomy group} for a foliation. For more details see for example \cite{BoThesis, BB}.

Consider a $q-$codimensional foliation $\mathcal{F}$ in a manifold $M$. Suppose $x \in M$, $y\in \mathcal{F}_{loc}(x)$ and $D_x, D_y$ are two small discs transversal to $\mathcal{F}$. The \textit{local holonomy map} $h^{\mathcal{F}}_{x,y}:D_x\to D_y$ is defined as the following: for any $z\in D_x$, the local leaves $\mathcal{F}_z$ intersect $D_y$ in exactly one point which we denote by $h^{\mathcal{F}}_{x,y}(z)$. 

Moreover, for any continuous path $\gamma:[0,1]\to M$ lies entirely inside a leaf $\mathcal{F}(x)$, we could define $h^\mathcal{F}_\gamma$ the \textit{holonomy along }$\gamma$: we take a subdivision $0=t_0<\cdots <t_n=1$  such that $|t_i-t_{i-1}|$ small enough and a sequence of small discs $D(x_i)\ni x_i $ which are transverse to $\mathcal{F}$, here $x_i:= \gamma(t_i)$. By definition of local holonomy map,  $h^{\mathcal{F}}_{x_{k-1},x_k}: D(x_{k-1})\to D(x_k)$ is well-defined on a neighborhood of $x_k\in D(x_k)$. Then the holonomy along $\gamma$ 
\begin{equation}\label{def holonomy along path}
h^\mathcal{F}_\gamma:=h^{\mathcal{F}}_{x_{n-1},x_n}\circ\cdots\circ h^{\mathcal{F}}_{x_0,x_1}: D(x_0)\to D(x_n)
\end{equation} 
is also well-defined on a neighborhood of $x_0\in D(x_0)$. 

Now for a $x\in M$, we consider all closed paths $\gamma$ lie in the leaf $\mathcal{F}(x)$ start from $x$ a small disc $D(x)\ni x$ transverse to $\mathcal{F}(x)$. By identifying $D(x)$ with $\RR^q$,  we get a group homomorphism $$\pi_1  (\mathcal{F}(x), x ) \to \mathrm{Homeo}  (\mathbb{R}^q,  0)$$
 where $\mathrm{Homeo} (\mathbb{R}^q,  0)$ is the set of germs of homeomorphisms $R^q\to R^q$ which fix the origin (since the germ of $h^\mathcal{F}_\gamma$ only depends on the homotopy class of $\gamma$). The image of the homomorphism is called the \textit{holonomy
group} of the leaf $\mathcal{F}(x)$  and denoted by $\mathrm{Hol}(\mathcal{F}(x), x)$. A leaf $\mathcal{F}(x)$ of a foliation $\mathcal{F}$ \textit{has trivial holonomy (group)} if the holonomy group $\mathrm{Hol}(\mathcal{F}(x), x )$ for any $x\in L$  is a trivial
group. A foliation $\mathcal{F}$ is called \emph{with trivial holonomy} (or \emph{uniformly compact}) if every leaf is compact and has trivial (or finite respect.) holonomy group. Notice that for a foliation with trivial holonomy, the holonomy along path is actually independ on of the choice of the path. Therefore in this case we can define holonomy map for any two points on the same leaf between two transverse manifolds.

In the rest of paper we will meet the holonomies induced by different dynamically defined foliation, for example center and center-(un)stable holonomies of partially hyperbolic diffeomorphisms.


\subsection{Regularity of foliations}\label{subsection: reg of fol}
For $r \geq 0$ we write that a map is $C^r$ if it is $C^{[r]}$ and the $[r]$th-order derivatives are uniformly H\"older continuous of exponent $r-[r]$ and $C^{r+}$ if it is $C^{r+\varepsilon}$ for some $\varepsilon> 0$. For a foliation $\mathcal{W}$ of an $n$-dimensional smooth manifold $M$ by $k$-dimensional submanifolds we define $\mathcal{W}$ to be a $C^{r}$ foliation if for each $x \in M$ there is an open neighborhood $V_{x}$ of $x$ and a $C^{r}$ diffeomorphism $\Psi_{x}: V_{x} \rightarrow D^{k} \times D^{n-k} \subset \RR^{n}$ (where $D^{j}$ denotes the unit ball  in $\RR^{j}$) such that $\Psi_{x}$ maps $\mathcal{W}$ to the standard smooth foliation of $D^{k} \times D^{n-k}$ by $k$-disks $D^{k} \times \{y\}$, $y \in D^{n-k}$, $(V_x, \Psi_x)$ is called a foliation box. This is the notion of regularity of a foliation considered by Pugh, Shub, and Wilkinson in \cite{Pugh-Shub-Wilkinson}. Roughly speaking a foliation $\mathcal{W}$ in $M$ is $C^r$ if we could find $C^r-$foliation atlas $\{V_x\}_{x\in M}$ of $\mathcal{W}$ on $M$.

\subsection{Absolute continuity of foliation }\label{subsection: abs cont}
For a submanifold $S$ of $M$ we let $\nu_{S}$ be the induced Riemannian volume on $S$ from $M$. Suppose $\mathcal{F}_{i}, i=1,2$ are $d_i-$dimensional  foliations of manifold $M$ respectively and $\mathcal{F}_1$ subfoliates $\mathcal{F}_2$. We define $\mathcal{F}_1$ to be \emph{strongly absolutely continuous} in $\mathcal{F}_2$ if for any pair of nearby smooth transversal $(d_2-d_1)$-dimensional submanifolds $S$ and $S'$ within $\mathcal{F}_2$ for $\mathcal{F}_1$ the $\mathcal{F}_1$-holonomy map $h^{\mathcal{F}_1}: S \rightarrow S'$ is absolutely continuous with respect to the measures $\nu_{S}$ and $\nu_{S'}$. 

Every $C^{1}-$foliation is strongly absolutely continuous in the manifold, and the strong (un)stable manifolds of a partially hyperbolic diffeomophism are strongly absolutely continuous. What we call strong absolute continuity is the classical notion of absolute continuity used before (cf. \cite{BS}, \cite{AVW}), we define a weaker notion of absolute continuity as the following.

\begin{definition}\label{defn: ac}
The foliation $\mathcal{F}_1$ is called \emph{absolutely continuous} in $\mathcal{F}_2$ if for each $x \in M$ there is an open neighborhood $V$ of $x$ and a  foliation $\mathcal{W}$ strongly absolutely continuous   in $\mathcal{F}_2$ transverse to $\mathcal{F}_1$ (within $\mathcal{F}_2$) such that for any pair of points $y,z \in V$ the $\mathcal{F}_1$-holonomy map $h^{\mathcal{F}_1}: \mathcal{W}(y) \rightarrow \mathcal{W}(z)$ is absolutely continuous with respect to the measures $\nu_{\mathcal{W}}(y), \nu_{\mathcal{W}}(z)$ respectively. 
\end{definition}

We consider another weaker definition of absolute continuity of foliation in which is useful in the study of partially hyperbolic dynamics (cf. \cite{BX} and \cite{AVW}): suppose $\mathcal{F}_{1,2}$ are two foliations of manifold $M$ and $\mathcal{F}_1$ subfoliates $\mathcal{F}_2$. We say the Riemannian measure $\nu_{\mathcal{F}_2}$ \textit{has Lebesgue disintegration along $\mathcal{F}_1$} if given any measurable
set $Y$ contained in a leaf $\mathcal{F}_2(x)$, then $\nu_{\mathcal{F}_2}(x)(Y) = 0$ if and only if for $\nu_{\mathcal{F}_2}(x)-$almost every $z \in M$ the leaf
$\mathcal{F}_1(z)$ intersects $Y$ in a zero $\nu_{\mathcal{F}_1(z)}$-measure set. In other words, for any local foliation box $B$ of $\mathcal{F}_1$ on an arbitrary leaf $\mathcal{F}_2(x)$, for almost every leaf $\mathcal{F}_1(z)$ in the box $B$, the measure $(\nu_{\mathcal{F}_2(x)})|_B$ has conditional measure (with respect to the partition formed by foliation $\mathcal{F}_1$) equivalent to the Riemannian measure on $\mathcal{F}_1(z)$. 

Roughly speaking, the definition of Lebesgue disintegration guarantees that "Fubini's Nightmare" does not occur on any leaf of $\mathcal{F}_2$. The following result is standard, cf.  Proposition 16. in \cite{BX} or Proposition 6.2.2. in \cite{BS}.  As before we assume $\mathcal{F}_{i},i=1,2$ are two foliations of the manifold $M$ and $\mathcal{F}_1$ subfoliates $\mathcal{F}_2$.

\begin{lemma}\label{lemma: transverse ac}
Suppose $\mathcal{F}_1$ is absolutely continuous in $\mathcal{F}_2$. Then $\nu_{\mathcal{F}_2}$ has Lebesgue disintegration along $\mathcal{F}_1$.  
\end{lemma}

\subsection{Bunching condition for partially hyperbolic diffeomorphisms and actions}\label{subsection: bunching cond}Let $f$ be a smooth
partially hyperbolic diffeomorphism then $W^{u,s}-$foliations have uniformly smooth leaves, but in general we have no smoothness for leaves of $W^{c,cs,cu}$. 

For $r > 0$, we say that $f$ is $r-$bunched
if there exists an integer $k\geq 1$ such that for any point $p\in M$:
\begin{eqnarray*}
\|D_pf^k|_{E^s}\|\cdot \|(D_pf^k_{E^c})^{-1}\|^r<1,\quad \|(D_pf^k|_{E^u})^{-1}\|\cdot \|D_pf^k|_{E^c}\|^r<1; \\
\|D_pf^k|_{E^s}\|\cdot \|(D_pf^k_{E^c})^{-1}\|\cdot \|D_pf^k|_{E^c}\|^r<1;\\
\|(D_pf^k|_{E^u})^{-1}\|\cdot \|D_pf^k|_{E^c}\|\cdot \|(D_pf^k|_{E^c})^{-1}\|^r<1
\end{eqnarray*}
We say $f$ is $\infty-$bunched if it is $r-$bunched for any $r>0$. Notice that every partially hyperbolic diffeomorphism is $r-$bunched, for
some $r > 0$ and $1-$bunching is the \textit{center bunching} condition considered by Burns and Wilkinson in their proof of
the ergodicity of accessible, volume-preserving, center-bunched $C^2$ partially hyperbolic
diffeomorphisms \cite{BW10}. When $f$ is $C^r$
and dynamically coherent (will be defined later), these inequalities ensure that the leaves of $W^{cu}$,
$W^{cs}$ and $W^c$ are uniformly $C^{r+}$, cf. \cite{Pugh-Shub-Wilkinson} or \cite{AVW}. A natural situation in which
the $\infty$-bunching condition holds is when there is a continuous Riemannian metric
on $E^c$ with respect to which $Df|_{E^c}$ is an isometry.

A partially hyperbolic action $\alpha$ is called $r-$bunched if each Weyl chamber contains a partially hyperbolic element $a$ such that $\al(a)$ is $r-$bunched with respect to the dominatted splitting $E^u_a\oplus E^c\oplus E^s_a$.

\begin{remark}In fact, for $\ZZ^k-$partially hyperbolic action, $k>1$, some inequalities in $r-$bunching condition may not be necessary, see the discussion about $r-$bunching condition in \cite{DX}. 
\end{remark}



\section{Preliminaries}\label{section: preliminar}
\subsection{Uniformly compact center foliation and dynamical coherence for partially hyperbolic dynamics}\label{subsection: uniform compa cent}
Bohnet and Bonatti in \cite{BB} proved a useful result for integrablity of center-(un)stable bundles of a partially hyperbolic system. We say a partially hyperbolic diffeomorphism $f$ is \emph{dynamically coherent} if there are $f$-invariant foliations $\mathcal{W}^{cs}$ and $\mathcal{W}^{cu}$ with $C^{1}$ leaves which are tangent to $E^{s} \oplus E^{c}$ and $E^{c} \oplus E^{u}$ respectively. 

By \cite{BW08}, $f$ has a center foliation if $f$ is dynamically coherent. Furthermore, $\mathcal{W}^c$ and $\mathcal{W}^u$ subfoliate $\mathcal{W}^{cu}$, $\mathcal{W}^c$ and $\mathcal{W}^s$ subfoliate $\mathcal{W}^{cs}$. The converse is false in general since there exists example of partially hyperbolic diffeomorphism with an integrable center bundle but not dynamical coherent \cite{HRHU10}. In \cite{BB}, the authors proved the converse is true under the assumption of uniform compacticity of center foliation. 

\begin{lemma}\label{lemma: uniform comp imply dyn coh} 
Suppose $f:M\to M$ is a partially hyperbolic diffeomorphism. If $W^c(f)$ exists and is uniformly compact then $f$ is dynamical coherent.
\end{lemma}
\subsection{Absolutely continuity, regularity of foliations and associated holonomy maps} 

The following lemma will be used several times later. 

\begin{lemma}\label{lemma: leb disintegration filtration}Suppose $\mathcal{F}_{i},i=1,2,3$ are foliations of manifold $M$ and $\mathcal{F}_1$ subfoliates $\mathcal{F}_2$, $\mathcal{F}_2$ subfoliates $\mathcal{F}_3$. If $\nu_{\mathcal{F}_3}$ has Lebesgue disintegration along $\mathcal{F}_2$ and $\nu_{\mathcal{F}_2}$ has Lebesgue disintegration along $\mathcal{F}_1$, then $\nu_{\mathcal{F}_3}$ has Lebesgue disintegration along $\mathcal{F}_1$.
\end{lemma}
\begin{proof}The proof is standard, for completeness we give the details here. We only need to prove Lemma \ref{lemma: leb disintegration filtration} locally. So without loss of generality we could assume the measurable partitions induced by the foliations are measurable. For any measurable set $Z$ contained in an arbitrary leaf $\mathcal{F}_1(x)$, if $\nu_{\mathcal{F}_1(x)}(Z)=0$, then since $\nu_{\mathcal{F}_1}$ has Lebesgue disintegration along $\mathcal{F}_2$, for almost every $\mathcal{F}_2(y)$ leaf in $\mathcal{F}_1(x)$, 
\begin{equation}\label{eqn: in proof Lemma Lebesgue}
\nu_{\mathcal{F}_2(y)}(\mathcal{F}_2(y)\cap Z)=0
\end{equation} Moreover since $\mu_{\mathcal{F}_2}$ has Lebesgue disintegration along $\mathcal{F}_3$, then for any leaf $\mathcal{F}_2(y)$ satisfying \eqref{eqn: in proof Lemma Lebesgue}, for $\nu_{\mathcal{F}_2(y)}-$almost every $z$, $\nu_{\mathcal{F}_3(z)}(\mathcal{F}_3(z)\cap Z)=0$, therefore since $\nu_{\mathcal{F}_1}$ has Lebesgue disintegration along $\mathcal{F}_2$, we have that for $\nu_{\mathcal{F}_1(x)}$ almost every $z$, $\nu_{\mathcal{F}_3(z)}(\mathcal{F}_3(z)\cap Z)=0$. 

On the other hand, if for $\nu_{\mathcal{F}_1(x)}$ almost every $z$, 
\begin{equation}\label{eqn: 2 in lemma leb}\nu_{\mathcal{F}_3(z)}(\mathcal{F}_3(z)\cap Z)=0
\end{equation}
we hope to prove $\nu_{\mathcal{F}_1(x)}(Z)=0$. Since $\nu_{\mathcal{F}_1}$ has Lebesgue disintegration along $\mathcal{F}_2$, we only need to prove for almost every leaf $\mathcal{F}_2(y)$, \eqref{eqn: in proof Lemma Lebesgue} holds. But we know the set of $z$ which satisfy \eqref{eqn: 2 in lemma leb} has full $\nu_{\mathcal{F}_2(y)}-$measure for almost every leaf $\mathcal{F}_2(y)$, and on such $\mathcal{F}_2(y)$ since $\nu_{\mathcal{F}_2(y)}$ has Lebesgue disintegration along $\mathcal{F}_3$, then by \eqref{eqn: 2 in lemma leb} we know \eqref{eqn: in proof Lemma Lebesgue} holds. Therefore we have $\nu_{\mathcal F_1(x)}(Z)=0$.

\end{proof}

To prove the regularity of dynamically defined foliation, a powerful tool is to consider the associated holonomy map.

\begin{lemma}\label{lemma: reg of holonomy and foliations}Suppose $\mathcal{W},\mathcal{F},\mathcal{L}$ are foliations of $M$, $\mathcal{W},\mathcal{F}$ subfoliate $\mathcal{L}$ and $\mathcal{W}$ transversal to $\mathcal{F}$ (within $\mathcal{L}$). And we set $h=h^{\mathcal{W}}$ the holonomy within $\mathcal{L}$ along $\mathcal{W}$ between the leaves of $\mathcal{F}$. If $\mathcal{F},\mathcal{W},\mathcal{L}$ have uniformly $C^{r+}$ leaves and $h$ is uniformly $C^{r+}$, then $\mathcal{W}$ is a $C^{r+}-$foliation within $\mathcal{L}$ (in the sense that on every leaf of $\mathcal{L}$ we can find a $C^r$ foliation atlas for $\mathcal{W}$).

\end{lemma}
\begin{proof}For the case $\mathcal{L}=M$, the proof basicially is an application of Journ\'e Lemma \cite{Journe} which can be found in \cite{BX}, \cite{Sadovskaya}, \cite{Pugh-Shub-Wilkinson}, etc. The same argument also works here.
\end{proof}

The following is another standard result for regularity of foliation:

\begin{lemma}\label{lemma: intersection regular}Suppose $\mathcal{F}_i, i=1,2$ are two transverse $C^r-$foliations in $M$. Then $\mathcal{F}_1\cap \mathcal{F}_2$ is also a $C^r$ foliation.
\end{lemma}
\begin{proof}Suppose that $\dim M=n, \dim\mathcal{F}_{i}=d_i, i=1,2, d_1+d_2>n$. We only need to prove that for any $x\in M$ there is a neighborhood $U_x$ of $x$ and a $C^r$ chart $(\varphi, U_x)$  such that $\varphi: U_x\to \RR^k$ maps $\mathcal{F}_i\cap U_x, i=1,2$ into two family of transverse $d_i-$dimensional affine subspaces in $\RR^k$.

By taking the foliation box of $\mathcal{F}_1$ we could assume that $\mathcal{F}_1\cap U_x$ to be the standard foliation of $D^{d_1}\times D^{n-d_1}$ by disks $D^{d_1}\times \{y\}, y\in D^{n-d_1}$,  where $D^k$ is the unit ball in $\RR^k$. Up to a linear transformation we assume that $$T\mathcal F_2(\mathbf{0})=\{x_1=\cdots = x_{n-d_2}=0\}$$where $\mathcal{F}_2(\bf{0})$ is leaf of $\mathcal{F}_2$ through $\mathbf{0}=(0,\dots,0)$. Then since $\mathcal{F}_2$ is a $C^r$ foliation by implicit function theorem we can find $C^r-$functions $\varphi_1,\cdots, \varphi_{n-d_2}$ such that locally we have the leaf of $\mathcal{F}_2$ through $(x_1,\dots, x_{n-d_2},0,\dots,0)$ has the form $$\{(x_1+\varphi_1(y_1,\dots, y_{n}), \dots, x_{n-d_2}+\varphi_{n-d_2}(y_1,\dots,y_{n}), y_{n-d_2+1},\dots, y_{n}), |y_i| \text{ small }\}$$where $(\varphi_1(\mathbf{0}),\dots, \varphi_{n-d_2}(\mathbf{0}))=(0,\dots,0)$.

Therefore we consider a $C^r-$chart $\Phi$ defined in a small neighborhood of $\bf 0$, 
\begin{eqnarray}\nonumber
\Phi(x_1,\dots, x_n):= (x_1-\varphi_1(x_{1},\cdots, x_n), \dots, x_{n-d_2}-\varphi_{n-d_2}(x_{1},\cdots, x_n), x_{n-d_2+1},\dots, x_n) 
\end{eqnarray}
Then under this new chart, locally $\mathcal{F}_{i}, i=1,2$ are affine subspaces in $\RR^k$ which implies Lemma \ref{lemma: intersection regular}.\end{proof}
%
We will use the following foliation version of Journ\'e  Lemma later (cf.\cite{Journe}):

\begin{lemma}\label{lemma: Foliation version Journe Lemma}
Suppose $\mathcal{F}_{i},i=1,2,3$ are foliations of a smooth manifold $M$ with uniformly $C^{r+}-$leaves. We assume $\mathcal{F}_{1,2}$ subfoliate $\mathcal{F}_3$ and $\mathcal{F}_1$ transverse to $\mathcal{F}_2$ within $\mathcal{F}_3$. Moreover $\mathcal{W}:=\mathcal{F}_1\cap\mathcal{F}_2$ is a $C^{r+}$ foliation within $\mathcal{F}_{i},i=1,2$ respectively, then $\mathcal{W}$ is a $C^{r+}$ foliation in $\mathcal{F}_3$.
\end{lemma}

\begin{proof}We only need to prove the lemma locally, suppose $\mathcal{F}_{1,2}$ are tranversal foliations in a region $U$ of $\RR^d$ (identified with a local leaf of $\mathcal{F}_3$), and $\dim\mathcal{F}_i=d_i, i=1,2$, then $\dim\mathcal{W}=d_1+d_2-d$.
 
For any $x\in U$, we define a $(2d-d_1-d_2)-$dimensional foliation $\mathcal{L}$ in a neighborhood $U_x$ of $x$. In addition  we let $\mathcal{L}$ to be uniformly transverse to $\mathcal{W}$ and $\mathcal{F}_{1,2}$. Moreover we assume that $\mathcal{L}$ have uniformly $C^{r+}-$leaves (for example, a family of affine discs in $U_x$). By Lemma \ref{lemma: reg of holonomy and foliations}, to prove $\mathcal{W}$ is a $C^{r+}-$foliation in $U_x$, we only need to prove the holonomy map $h^\mathcal{W}$ induced by $\mathcal{W}$ between arbitrary two leaves of $\mathcal{L}$ is uniformly $C^{r+}$. 

By the transversality assumptions, $\mathcal{L}_i:=\mathcal{L}\cap \mathcal{F}_i,i=1,2$ are subfoliations of $\mathcal{L}$ transverse to each other. Since $\mathcal{F}_i,\mathcal{L}$ have uniformly $C^{r+}-$leaves, $\mathcal{L}_i$ also have uniformly $C^{r+}-$leaves. By our condition of lemma, $\mathcal{W}$ is $C^{r+}-$foliations in $\mathcal{F}_i$. As a corollary,  $h^\mathcal{W}$ is uniformly $C^{r+}$ along $\mathcal{L}_i$.  By Journ\'e lemma (cf. \cite{Journe}), $h^\mathcal{W}$ is also uniformly $C^{r+}$ on $\mathcal{L}$.

Let $U_x$ run over all $x\in U$, we get the proof.
\end{proof}
\subsection{Uniform quasiconformality and regularity of invariant foliations for partially hyperbolic dynamics}
Sadovskaya in \cite{Sadovskaya} proved a useful lemma for a non stationary linearization of dynamics, we consider the following version that we will use soon.
\begin{lemma}\label{lemma Sadov linearization}Suppose that $f:M\to M$ is a diffeomorphism on a Riemmanian manifold $M$ and $E\subset TM$ is a $f-$invariant subbundle such that $f$ is uniformly contracting (or expanding) and uniformly quasiconformal  on $E$. Moreover we assume there is a foliation $\mathcal{W}$ tangents to $E$ with $C^\infty$ leaves. Then for each $x \in M$ there is a $C^{\infty}$ diffeomorphism $\Phi_{x}: E_{x} \rightarrow \mathcal{W}(x)$ satisfying
		\begin{enumerate}
			\item $\Phi_{f(x)} \circ Df_{x} = f \circ \Phi_{x}$,
			
			\item $\Phi_{x}(0) = x$ and $D_{0}\Phi_{x}$ is the identity map,
			
			\item The family of diffeomorphisms $\{\Phi_{x}\}_{x \in M}$ varies continuously with $x$ in the $C^{\infty}$ topology. 
		\end{enumerate}
		The family $\{\Phi_{x}\}_{x \in M}$ satisfying (1), (2), and (3) is unique. 
\end{lemma}




In \cite{BX}, the authors proved the regularity of center foliation for a partially hyperbolic diffeomorphism under some quasiconformal assumptions.   Suppose $f:M\to M$ is a partially hyperbolic dynamical diffeomorphism with uniformly compact center foliation and $TM=E^u\oplus E^c\oplus E^s$ is the dominated splitting associated. We assume that there exist $f-$invariant subbundles $\tilde{E}^u\subset E^u$, $\tilde{E}^s\subset E^s$ which tangent to some $f-$invariant foliations $\tilde{\mathcal{F}}^{u}$ and $\tilde{\mathcal{F}}^s$ repectively. In addition we assume that $\mathcal{F}^{s,u}$ have $C^\infty$ leaves and $f$ is uniformly quasiconformal on $\tilde{E}^u$ ($\dim(\tilde{E}^u)>1$) . Then by Lemma \ref{lemma Sadov linearization}, there is a unique $C^\infty$ family $\{\Phi_x\}_{x\in M}$ for $\tilde{\mathcal{F}}^{u}$ satisfying all conditions in Lemma \ref{lemma Sadov linearization}. We have the following proposition proved in \cite{BX}.

\begin{prop}\label{prop center holonomy quasiconf}
i) If $\tilde{E}^u\oplus E^c$ is integrable and tangents to a foliation $\tilde{\mathcal{F}}^{cu}$ then there is a constant $K \geq 1$ such that for any two points $x \in M$, $y \in W^{c}_{loc}(x)$, the homeomorphism 
		\[
		\Phi_{y}^{-1} \circ h^{c} \circ \Phi_{x}:  \Phi_{x}^{-1}(\tilde{\mathcal{F}}^{u}_{loc}(x)) \rightarrow \Phi_y^{-1}(\tilde{\mathcal{F}}^{u}_{loc}(y))
		\]
is $K$-quasiconformal. Here $h^{c}$ is the center holonomy within $\tilde{\mathcal{F}}^{cu}$, between two $\tilde{\mathcal F}^u-$leaves. In particular the leaf Riemannian measure on $\tilde{\mathcal{F}}^{cu}$ has Lebesgue disintegration on $W^c$.

ii). If $\tilde{E}^u\oplus E^c\oplus \tilde{E}^s$ is integrable and tangents to a foliation $\tilde{\mathcal{F}}^{cus}$ then there is a constant $K' \geq 1$ such that for any two points $x \in M$, $y \in \tilde{\mathcal{F}}^{cs}_{loc}(x)$, the homeomorphism 
		\[
		\Phi_{y}^{-1} \circ \tilde{h}^{cs} \circ \Phi_{x}:  \Phi_{x}^{-1}(\tilde{\mathcal{F}}^{u}_{loc}(x)) \rightarrow \Phi_y^{-1}(\tilde{\mathcal{F}}^{u}_{loc}(y))
		\]
is $K'$-quasiconformal. Here ${h}^{cs}$ is the center stable holonomy within $\tilde{\mathcal{F}}^{cus}$, between two $\tilde{\mathcal F}^u-$leaves. As a corollary, the leaf Riemannian measure on $\mathcal{F}^{cus}$ has Lebesgue disintegration on $\tilde{\mathcal{F}}^{cs}$.
\end{prop}
\begin{proof}In Lemma 12 of \cite{BX} the authors proved Proposition \ref{prop center holonomy quasiconf} for the case $\tilde{E}^u=E^u, \tilde{E}^s=E^s$ (notice that the holonomy of $W^c-$foliation between local $W^u-$leaves inside of $W^{cu}(x)$ coincides with the center-stable holonomy in $M$), the proof here is the same.
\end{proof}

Another useful result for regularity of center foliation in \cite{BX} is the following: suppose $f:M\to M$ is a $C^\infty$ volume preserving partially hyperbolic diffeomorphism on manifold $M$. We assume  $W^c(f)$ is compact and has trivial holonomy. In addition, $f$ is uniformly quasiconformal on $E^u(f)$ and $\dim E^u>1$. We denote by $h^c$ the center holonomy within the center-unstable foliation $W^{cu}(f)$ between two unstable leaves and $h^{cs}$ the center-stable holonomy between two unstable leaves.
\begin{prop}\label{prop: ac implies linearity of center fol}If the volume $\nu$ on $M$ has Lebesgue disintegration on $W^c$, then $h^c, h^{cs}$ are uniformly smooth.  Moreover $W^{cs}$ is a $C^{r+}-$foliation and $W^c$ is a $C^{r+}-$foliation within $W^{cu}$ if $f$ is $r-$bunched for some $r\geq 1$.
\end{prop}
\begin{proof}The proof is the same as the proof of smoothness of center-stable holonomy in \cite{BX} (cf. Corollary 30. of \cite{BX}). In \cite{BX} the authors assume that $f$ is also uniformly quasiconformal on $E^s(f)$, but this condition is only used to prove $\nu$ has Lebesgue disintegration on $W^c$.
\end{proof}

\section{Proof of Theorem \ref{main}: $\dim E^i>1$ for every $i$.}\label{section: proof of thm 1, dim>1}

\subsection{Lebesgue density along center foliation}
\begin{prop}\label{prop: Leb dis inte along Wc}Under assumptions of Theorem \ref{main}, if $\dim W^i>1$ for any $i$, then $\nu_M$ has Lebesgue disintegration along $W^c$.
\end{prop}
\begin{proof}
We consider the coarse Lyapunov splitting of $\al$, $TM=E^1\oplus\cdots \oplus E^n\oplus E^c$. By assumption A. we could find  $a_1\in A$ such that $\al(a_1)$ is expanding on  $E^1$, uniformly quasiconformal on $E^1$ and and contracts $E^2\oplus \cdots \oplus E^n$. In other words we have the partially hyperbolic splitting for $\alpha(a_1)$: $$E^u(\al(a_1))=E^u_{a_1}, E^c(\al(a_1))=E^c, E^s(\al(a_1))=E^2\oplus\cdots\oplus E^n$$

Since $\mathcal{W}^c$ is uniformly compact, by Lemma \ref{lemma: uniform comp imply dyn coh}, $E^1\oplus E^c$, $E^2\oplus \cdots\oplus E^n$ and $E^2\oplus \cdots \oplus E^n\oplus E^c$ are integrable. We assume $E^1, E^1\oplus E^c $, $E^2\oplus \cdots\oplus E^n$ and $E^2\oplus \cdots E^n\oplus E^c$ tangents to foliations $W^1, W^{1,c}, W^{2,\dots,n}, W^{2,\dots,n,c}$ respectively. By theory of partially hyperbolic dynamics we know $W^1$ and $W^{2,\dots, n}$ have smooth leaves. Let $\tilde{E}^{u}, E^c, \tilde{E}^{s}$ in (ii) of Proposition \ref{prop center holonomy quasiconf} be $E^1, E^c, E^{2,\dots,n}$  here, by (ii). of Proposition \ref{prop center holonomy quasiconf} we know the volume on $M$ has Lebesgue disintegration along $W^{2,\dots,n,c}$.

Now we consider $E^2$, by assumption A. again we could find $a_2\in A$ such that $\al(a_2)$ expands $E^2$, uniformly quasiconformal on $E^2$ and and contracts $E^1\oplus E^3\oplus \cdots \oplus E^n$. As previous discussion for $a_1$, $E^1\oplus E^3\oplus\cdots\oplus E^n\oplus E^c$ is integrable, tangents to a foliation $W^{1,3,\dots,n,c}$. Then $E^3\oplus \cdots \oplus E^n\oplus E^c$ and $E^3\oplus \cdots \oplus E^n$ are integrable and tangent to the foliations $$W^{3,\dots, n,c}:=W^{2,\dots,n,c}\cap W^{1,3,\dots,n,c}, \quad W^{3,\dots, n}:=W^{2,\dots,n}\cap W^{1,3,\dots,n}$$
respectively. Consider $\al(a_2)$'s action on $M$, let 
$\tilde{E}^{u}, E^c, \tilde{E}^{s}$ in (ii) of Proposition \ref{prop center holonomy quasiconf} be $E^2, E^c, E^{3,\dots,n}$ in Theorem \ref{main}, using ii). of Proposition \ref{prop center holonomy quasiconf} again we have that $\nu_{W^{2,\dots,n,c}}$ has Lebesgue disintegration along $W^{3,\dots,n,c}$.

Repeat the same argument by induction we can prove that for any $i>2$, $E^i\oplus\cdots\oplus E^n\oplus E^c$ is integrable, tangents to foliation $$W^{i,\dots,n,c}:=W^{i-1,\dots,n,c}\cap W^{i-2,i,\dots,n,c}$$ and $\nu_{W^{i-1,\dots,n,c}}$ has Lebesgue disintegration along $W^{i,\dots,n,c}$. In addition for $E^n$ by assumption A. in Theorem \ref{main} we can pick an $a_n\in A$ such that $\al(a_n)$ expands $E^n$, uniformly quasiconformal on $E^n$ and contracts $E^1\oplus \cdots \oplus E^{n-1}$. By Lemma \ref{lemma: uniform comp imply dyn coh} we could define $W^{n,c}$. Moreover by (i). of Proposition \ref{prop center holonomy quasiconf}, $\nu_{W^{n,c}}$ has Lebesgue disintegration along $W^c$.

In summary, we get the following filtration of foliation $$W^c\subset W^{n,c}\subset W^{n-1,n,c}\subset \cdots\subset W^{2,\dots,n,c}\subset W^{1,\dots,n,c}:=M$$
and the Riemannian measure on each foliation has Lebesgue disinegration along the former one. Using Lemma \ref{lemma: leb disintegration filtration} $n$ times we get the proof.
\end{proof}

As a corollary of the proof of Proposition \ref{prop: Leb dis inte along Wc}, we have
\begin{coro}\label{coro: all fol integrable}Consider the set of indices $I=\{1,\dots,n,c\}$, for any non-empty subset $J\subset I$, $J\neq\{1,\dots, n\}$, we have $\oplus_{j\in J}E^j$
is integrable. In particular $$E^{\hat{i}}:=E^{1,\dots, i-1,i+1,\dots,n},E^{\hat{i},c}:=E^{1,\dots, i-1,i+1,\dots, ,c} $$ are all integrable. Moreover $W^i$ and $W^{\hat{i}}:=W^{1,\dots,i-1,i+1,\dots,n}$ have uniformly $C^\infty$ leaves, $W^{i,c}$ and $W^{\hat{i},c}:=W^{1,\dots,i-1,i+1,\dots,n,c}$ have uniformly $C^{r+}$ leaves if $\alpha$ is $r-$bunched for some $r\geq 1$.
\end{coro}
\begin{proof}For any $i$ by assumption of Theorem \ref{main} there exists $a_i\in A$ such that $\al(a_i)$ expands $E^i$ and contracts $E^{\hat{i}}$.  Therefore we have the partially hyperbolic splitting:
$$E^u(\al(a_i))=E^u_{a_i}, E^c(\al(a_i))=E^c, E^s(\al(a_i))=E^{\hat{i}}$$
Then by Lemma \ref{lemma: uniform comp imply dyn coh} $E^{\hat{i}}$ and $E^{\hat{i},c}$ are both integrable. Moreover by discussion in section \ref{subsection: bunching cond}, $W^i, W^{\hat{i}}$ both have uniformly $C^\infty-$leaves and $W^i$, $W^{\hat{i},c}$ have uniformly $C^{r+}-$leaves if $\al$ is $r-$bunched for $r\geq 1$. Then for any non-empty subset $J\subset I, J\neq \{1,\dots, n\}$, $$\oplus_{j\in J}E^j=\begin{cases}\cap_{i\notin J}E^{\hat{i},c},~~\text{if $c\in J$}\\ \cap_{i\notin J}E^{\hat{i}},~~\text{if $c\notin J$}
\end{cases}$$
hence we get integrability for $\oplus_{j\in J}E^j$. 
\end{proof}
\subsection{Regularity of center foliation and proof of Theorem \ref{main}}
Under assumptions of Theorem \ref{main}, by Corollary \ref{coro: all fol integrable} we could define $W^{i,c}$ and $W^{\hat{i},c}$ for any $i$. We denote by $h^{i,c}$ the center holonomy within $W^{i,c}$ between two $W^i$ leaves and $h^{\hat{i},c}$ the holonomy induced by $W^{\hat{i},c}$ leaves between two $W^i$ leaves. 
Now we consider $\alpha(a_i)$  as in the proof of Proposition \ref{prop: Leb dis inte along Wc} and Corollary \ref{coro: all fol integrable}, by Proposition \ref{prop: Leb dis inte along Wc} we know $\al(a_i)$ satisfying all assumptions of Proposition \ref{prop: ac implies linearity of center fol}. Therefore we have 
\begin{prop}\label{prop: smoothness of foliations}$W^{\hat{i},c}$ is a $C^{r+}$ foliation of $M$ if $\al$ is $r-$bunched for some $r\geq 1$.
\end{prop}
\begin{proof}Apply Proposition \ref{prop: ac implies linearity of center fol} to the diffeomorphism $\al(a_i)$, since $W^u_{a_i}=W^i, h^c(\al(a_i))=h^{\hat{i},c}$ we know $h^{\hat{i},c}$ is uniformly $C^{\infty}$. By Corollary \ref{coro: all fol integrable}, $W^{i}$ have uniformly $C^\infty$ leaves  and $W^{\hat{i},c}$ have uniformly $C^{r+}$ leaves. Apply Lemma \ref{lemma: reg of holonomy and foliations} with $\mathcal{L}:=M, \mathcal{W}:=W^{\hat{i,c}}, \mathcal{F}:=W^{i}$, $W^{\hat{i},c}$ is a $C^{r+}$ foliaiton of $M$.
\end{proof}
As a result, we can now prove Theorem \ref{main} under assumption that for any $i$, $\dim(E^i)>1$. Suppose $\al$ is $r-$bunched, then by Proposition \ref{prop: smoothness of foliations} $W^{\hat{i},c}$ are $C^{r+}$ foliation every $i$. Then by applying Lemma \ref{lemma: intersection regular} repeatedly, we get that $W^c$ is also a $C^{r+}$ foliation. Since $W^c$ has trivial holonomy, then $N:=M/W^c$ is a $C^{r+}-$manifold, and the canonical projection $\pi: M\to N$ modulo $W^c$ is a $C^{r+}-$submersion and hence a fiber-bundle. Notice that the action $\al$ preserves $W^c-$foliation, therefore $\al$ induces an $C^{r+}$ $\ZZ^k-$action $\bar\al$ on $N$. So $\al$ is a $C^r-$fiber bundle extension over $\bar\al$.  Obviously for any regular $a\in A$, $\bar\al(a)$ is an Anosov diffeomorphism on $N$ then $\bar\al$ is an Anosov action on $N$. Therefore we constructed $N, \bar\al, \pi$ satisfing all conditions in Theorem \ref{main}. 

In particular if $r=\infty$, $W^c$ is a smooth foliation, therefore $N,\pi, \bar{\al}$ are all smooth. Moreover $\bar{\al}$ is a smooth Anosov action on $N$ and also satisfying assumption A. in Theorem \ref{main} in the sense that we still have the coarse Lyapunov splitting of $TN$: $$TN:=\bar{E}^1\oplus\cdots\oplus \bar{E}^n, \quad \bar E^i:=d\pi(E^i)$$
And for each $i$, there exists $a_i\in A$ (exactly the same $a_i$ for $\al$) such that $\bar{\al}(a_i)$ is uniformly quasiconformal on $\bar E^i$ and $\bar E^i=E^u_{a_i}$. 

Now we claim that $\bar{W}^{\hat{i}}=\pi (W^{\hat{i}})$ is a smooth foliation in $N$, which implies that each $\bar{W}^i:=\pi(W^i)$ is a smooth foliation (by applying Lemma \ref{lemma: intersection regular} several times). Take a foliation box $$\varphi:  U\to \RR^{\dim E^{\hat{i},c}}  \times \RR^{\dim E^i}, U\subset M$$  where maps the foliation $W^{\hat{i},c}\cap U$ to the standard foliation $ \RR^{\dim E^{\hat{i},c} }\times \{y\}, y\in \RR^{\dim E^i}$ (see section \ref{subsection: reg of fol}). Since $W^c$ is a smooth foliation, up to a smooth changing variable we could assume $\varphi$ maps $W^c\cap U$  to the standard foliation $$\RR^{\dim E^c}\times \{y\}\subset \RR^{\dim E^{\hat{i},c}}\subset  \RR^{\dim E^{\hat{i},c} }\times \RR^{\dim E^i}, y\in \RR^{\dim{E^{\hat{i}}}}$$ Then $\varphi$ induces a smooth map $\bar\varphi$ from $\pi(U)$ to $\RR^{\dim E^{\hat{i}}}\times \RR^{\dim E^i}\cong \RR^{\dim M}/\RR^{\dim E^c}$ which maps $\bar{W}^{\hat{i}}$ to the standard foliation $\RR^{\dim E^{\hat{i}}}\times \{y\}, y\in \RR^{\dim E^i}$. Therefore $(\bar\varphi,\pi(U))$ is a smooth foliation box for $\bar{W}^{\hat{i}}$. In summary $\bar{W}^{\hat{i}}=\pi (W^{\hat{i}})$ is a smooth foliation in $N$.


Suppose $\dim E^c=1$, we need to prove $\al$ is $\infty-$bunched and an isometric extension over $\bar\al$. The proof is basically the same as Lemma 33. \cite{BX}. For completeness we give an outline here. Since $\dim E^c=1$, $\al$ is center-bunched ($(1+)-$bunched), by previous arguments we know $W^c$ is a $C^{1+}$ foliation. Then $W^c$ has a continuous conditional density on the leaves which is invariant under $\alpha-$action. Thanks to $\dim E^c=1$ and the conditional density on $W^c$, we could rescale the metric on $E^c$ to get an $\alpha-$invariant $C^0$ metric on $E^c$. Then by the discussion in section \ref{subsection: bunching cond}, any non-trivial alement of $\alpha$ is $\infty-$bunched which implies $\alpha$ itself is $\infty-$bunched. Moreover $W^c$ is a $C^\infty$ foliation with a smooth $\al-$invariant conditional density on the leaves. Then by the same argument as before we could rescale the metric on $E^c$ to get an $\al-$invariant smoothly varing family of Riemannian metrics $\{d_x\}$ on the center leaves. Therefore $\al$ is an isometric extension over $\bar\al$.
\section{Proof of Theorem \ref{main}:  $\dim E^i=1$ for some $i$}\label{section: thm main dim ei=1}
In this section we prove Theorem \ref{main} when there exists some one dimensional $E^i$. Firstly, Corollary \ref{coro: all fol integrable} still holds. The key to prove Theorem \ref{main} is the following proposition.
\begin{prop}\label{prop: non-pathologic imply Wc good in Wic}Under assumptions of Theorem \ref{main}, if $W^c$ is non-pathological and $\al$ is $r-$bunched, then $W^c$ is a $C^{r+}$ foliation within $W^{i,c}$
\end{prop}

Suppose Proposition \ref{prop: non-pathologic imply Wc good in Wic} is true. Consider foliations $W^{1,c}, W^{2,c}, W^{1,2,c}$, by Corollary \ref{coro: all fol integrable}, they are all integrable and have $C^{r+}$ leaves if $\al$ is $r-$bunched. Moreover by Proposition \ref{prop: non-pathologic imply Wc good in Wic}, $W^c=W^{1,c}\cap W^{2,c}$ is a $C^{r+}$ foliation within $W^{1,c}$ and $W^{2,c}$. By Lemma \ref{lemma: Foliation version Journe Lemma}, $W^c$ is also a $C^{r+}$ foliation within $W^{1,2,c}$.

Now we consider $W^{1,2,c}, W^{3,c},W^{1,2,3,c}$, by the same argument we conclude that $W^c$ is also a $C^{r+}$ foliation within $W^{1,2,3,c}$. 

Apply the same method repeatly we can prove that $W^c$ is a $C^r$ folation of the whole manifold $M$. And the rest of the proof of Theorem \ref{main} (construction of $N,\bar{\al}, \pi$, the case of $r=\infty$ of $\dim(W^c)=1$) is the same as the case when $\dim(E^i)>1$ for each $i$.

The rest of the section is for the proof of Proposition \ref{prop: non-pathologic imply Wc good in Wic}.
\subsection{Proof of Proposition \ref{prop: non-pathologic imply Wc good in Wic}}
For $i$ such that $\dim E^i>1$, consider $a_i\in A$ such that $\al(a_i)$ satisfying assumption A. in Theorem \ref{main}. By Proposition \ref{prop: ac implies linearity of center fol}, $W^c$ is a $C^{r+}$ foliation within $W^{i,c}$.

Now we consider $i$ such that $\dim E^i=1$, by assumption A. we can take an $a_i\in A$ such that $E^u(\al(a_i))=E^i$. As before we denote by $h^{i,c}$ the center holonomy within $W^{i,c}$ between two $W^i$ leaves. By our assumption we have that for any two local $W^i$ leaves $W^i_{loc}(x), W^i_{loc}(y)$ within the same $W^{i,c}$ leaf, the map $$h^{i,c}_{xy}:W^i_{loc}(x)\to W^i_{loc}(y)$$ is absolutely continuous. Then Proposition \ref{prop: non-pathologic imply Wc good in Wic} is a direct corollary of the following proposition, Lemma  \ref{lemma: reg of holonomy and foliations} and $r-$bunching condition.

\begin{prop}\label{prop: single diff, non-pathologic imply smoothness}Suppose $f:M\to M$ is a smooth volume preserving partially hyperbolic diffeomorphism and $\dim E^u=1$. We assume that $W^c(f)$ is compact and with trivial holonomy. Moreover $\nu_M$ and has Lebesgue disintegration along $W^c$. If for any two local unstable leaves $W^u_{loc}(x), W^u_{loc}(y)$ in the same center-unstable leaf,  the center holonomy $h^c_{x,y}:W^u_{loc}(x)\to W^u_{loc}(y)$ is absolutely continuous, then $h^c$ is uniformly smooth on any local unstable leaf.
\end{prop}

\begin{proof}The proof is similar to Chapter 4. (the proof of Lemma 27.) of \cite{BX}. For completeness we give an outline here.

\textbf{Step 1: construction of a fiber space $\mathcal{E}$.} We now consider the space 
\[
\mathcal{E}=\{(x,y)\in M^2:  y\in W^c(x) \},
\]
and we define $F: \mathcal{E} \rightarrow \mathcal{E}$ by $F(x,y) = (f(x),f(y))$. Since $W^c$ is compact and with trivial holonomy, we can easily prove that $\mathcal{E}$ is a continuous fiber bundle over $M$ with compact fibers.  By considering the measurable partion of $M$ into $M=\cup_{x\in M}W^c(x)$ we get the family of conditional measure $\{m^c(x)\}$ of $\nu_M$ on $W^c(x) $ determined  by the partition. Since $f$ preserves $\nu_M$, we have $f_\ast m^c(x)=m^c(f(x))$ for $\nu_M-a.e. x\in M$. Moreover the measures $m^{c}(x)$ are equivalent to the Riemannian volume on $W^c$ since $\nu_M$ has Lebesgue disintegration along $W^{c}$ leaves. We define the measure $\mu$ on the fiber bundle $\mathcal{E}$ by setting, for any measurable set $A\subset \mathcal{E}$, 
\[
\mu(A)=\int \mathbf{1}_A(x,y)dm^c(x)(y) d\nu_M(x),
\]
where $\mathbf{1}_{A}$ denotes the characteristic function of $A$. Obviously $\mu$ is $F-$invariant.

\textbf{Step 2: $H^c_{\cdot,\cdot}:=Dh^c_{\cdot,\cdot}$ is well-defined $\mu-a.e.$. }Since $W^c$ has trivial holonomy, $h^c_{x,y}: W^u_{loc}(x)\to W^u_{loc}(y) $ is well-defined for any $(x,y)\in \mathcal{E}$, the following lemma shows that the derivative of $h^c$ is well-defined $\mu-a.e.$. 
\begin{lemma}\label{lemma: H^c mu ae welldef}

Let\begin{align*}
Q: = \{x \in M: \text{for $\nu_{W^c}(x)$-a.e. $y \in W^c$}, \
\text{$h^{c}_{x,y}:W^{u}_{loc}(x) \rightarrow W^{u}_{loc}(y)$ is differentiable at $x$}\}.
\end{align*}
Then $\nu_M(M-Q)=0$.
\end{lemma}
\begin{proof}See Proposition 23. in \cite{BX} (notice that here for any two local unstable leaves $W^u_{loc}(x), W^u_{loc}(y)$ in the same center-unstable leaf, $h^c_{x,y}:W^u_{loc}(x)\to W^u_{loc}(y)$ is absolutely continuous).\end{proof}




As a consequence, by definition of $\mu$, since $m^c(x)$ is equivalent to $\nu_{W^c}(x )$ for $\nu_M-a.e.~x$, we have that $h^c_{x,y}$ is well-defined for $\mu-a.e. ~(x,y)\in \mathcal{E}$. We denote $H^c_{x,y}:=Dh^c_{x,y}$, then $H^c$ is a $\mu-a.e.$ defined (non-zero) map on $\mathcal{E}$.

\textbf{Step 3: (Un)stable \textit{Holonomies} of cocycle $Df|_{E^u}$.} In this step we consider a family of linear map from $E^u(x)$ to $E^u(y)$ for any $x,y$ in the same local (un)stable leaves. We call it \textbf{Holonomy} of cocycle $Df|_{E^u}$.

The bundle $E^{u}$ is a H\"older continuous subbundle of $TM$  with some H\"older exponent $\beta > 0$ \cite{Pugh-Shub-Wilkinson}. Therefore the restriction $Df|_{E^{u}}$ is a H\"older continuous linear cocycle over $f$ in the sense of \cite{KS13}. For $x,y \in M$ two nearby points we let $I_{xy}: E^{u}(x) \rightarrow E^{u}(y)$ to be the unique linear identification such that $v, I_{xy}(v)$ same length and direction (locally it always works since $\dim{E^u}=1$). Obviously $I_{xy}$ is $\beta$-H\"older close to the identity. And the cocycle $Df|_{E^{u}}$ is \textit{fiber-bunched} since $\dim E^u=1$ (cf. \cite{KS13}). The following proposition thus applies to $Df|_{E^{u}}$.

\begin{lemma}\cite[Proposition 4.2]{KS13}\label{lemma: Holonomies}
For $y \in \mathcal{W}^{u}_{loc}(x)$, the limit 
\[
\lim_{n \rightarrow \infty}Df^{n}_{f^{-n}y} \circ I_{f^{-n}xf^{-n}y} \circ Df^{-n}_{x} |_{E^{u}} :=H^{u}_{xy},
\]
exists uniformly in $x$ and $y$ and defines a linear map from $E^{u}_{x}$ to $E^{u}_{y}$ with the following properties for $x,y,z \in M$,
\begin{enumerate}
\item $H^{u}_{xx} = Id$ and $H^{u}_{yz} \circ H^{u}_{xy} = H^{u}_{xz}$;
\item $H^{u}_{xy} = Df^{n}_{f^{-n}y} \circ H^{u}_{f^{-n}xf^{-n}y} \circ Df^{-n}_{x}$  for any $n \geq 0$. 
\item $\| H^{u}_{xy} - I_{xy}\| \leq Cd(x,y)^{\beta}$, $\beta$ the exponent of H\"older continuity for $E^{u}$.  
\end{enumerate}
Furthermore $H^{u}$ is the unique collection of linear identifications with these properties. Similarly if $y \in \mathcal{W}^{s}_{loc}(x)$ then the limit $\lim_{n \rightarrow \infty}Df^{-n}_{y} \circ I_{f^{n}xf^{n}y} \circ Df^{n}_{x} |_{E^{u}} :=H^{s}_{xy}$ exists and gives a linear map from $E^{u}_{x}$ to $E^{u}_{y}$ with analogous properties. $H^{u}$ and $H^{s}$ are known as the \emph{unstable} and \emph{stable Holonomies} of $Df|_{E^{u}}$ respectively. 
\end{lemma}

Since $f$ is uniformly quasiconformal on $E^u$, by Lemma \ref{lemma Sadov linearization} there is a family of coordinates $\{\Phi_x\}_{x\in M}$ satisfying all the properties of Lemma \ref{lemma Sadov linearization}. Moreover we have the following property:

\begin{lemma}[Proposition 9 of \cite{BX}]\label{lemma: affine transitions}
For each $x \in M$ and $y \in \mathcal{W}^{u}(x)$ the map $\Phi_{y}^{-1} \circ \Phi_{x}: E^{u}_{x} \rightarrow E^{u}_{y}$ is an affine map with derivative $H^{u}_{xy}$. 
\end{lemma}

\textbf{Step 4: $H^c$ is equivariant with $H^{u(s)}$}
The key to prove smoothness of $h^c$ is the following Lemma:

\begin{lemma}\label{lemma: holo eqva}There is a full $\mu$-measure subset $\Omega$ of $\mathcal{E}$ such that if $(x,y)$, $(z,w) \in \Omega$ with $z \in W^{u}_{loc}(x)$ and $w \in W^{u}_{loc}(y)$ then the following equation holds,
\begin{equation}\label{equation: cu holonomy}
H^{c}_{z,w}\circ H^u_{xz}=H^u_{yw}\circ H^c_{x,y},
\end{equation}
and similarly if $(x,y),(z,w) \in \Omega$ with $z \in W^{s}_{loc}(x)$ and $w \in W^{s}_{loc}(y)$ then,
\begin{equation}\label{equation: cs holonomy}
H^{c}_{z,w}\circ H^s_{xz}=H^s_{yw}\circ H^{c}_{x,y}.
\end{equation}
\end{lemma}
\begin{proof}The proof is basically the same as that of Lemma 24 in \cite{BX}. For any $x,y\in M$ and a linear map $H: E^u(x)\to E^u(y)$, we denote by $|H|:=\frac{\|H\cdot v\|}{\|v\|}, v\in E^u(x)-\{0\}$. Obviously $H$ is well-defined since $\dim E^u=1$. Since $W^c$ has trivial holonomy, the bundle $E^u$ along $W^c$ is orientable and $h^c$ between two unstable leaves preserving the orientation. Therefore we could replace $\det(H)$ in \cite{BX} by $|H|$ here, the rest of the proof is exactly the same.
\end{proof}

\textbf{Step 5: Linearity of $h^c$.}
The final step is to prove $h^c$ is actually a linear map on the charts $\{\Phi_x\}_{x\in M}$. The proof is similar to Lemma 27 in \cite{BX}. In fact, since $D_0\Phi_x=id$, we have $$D_{0}(\Phi_{y}^{-1} \circ h^{c}_{x,y} \circ \Phi_{x}) = H^{c}_{x,y}$$for any $(x,y)\in \Omega$. But by Lemma \ref{lemma: holo eqva}, \ref{lemma: affine transitions}, for any $(x,y)\in \Omega$ and any $v\in E_x^u$ such that $(\Phi_x(v), h^c_{x,y}(\Phi_x(v)))\in \Omega$, as in \cite{BX} we can also get that 
$$
D_{v}(\Phi_{y}^{-1} \circ h^{c}_{x,y} \circ \Phi_{x}) = H^{c}_{x,y},
$$
Then by absolute continuity of $h^c_{x,y}: W^u(x)\to W^u(y)$ we have that for any $(x,y)\in \Omega$, $\Phi_{y}^{-1} \circ h^{c}_{x,y} \circ \Phi_{x}$ is a $C^{1}$ map with derivative $H^{c}_{x,y}$ everywhere on $\Phi_x^{-1}(W_{loc}^u(x))$, i.e., $\Phi_{y}^{-1} \circ h^{c}_{x,y} \circ \Phi_{x}$ coincides exactly with the linear map $H^{c}_{x,y}$.  

By uniform continuity of $\Phi_{y}^{-1} \circ h^{c}_{x,y} \circ \Phi_{x}$ on the pair $(x,y)$ we conclude $H^c_{x,y}$ also depends uniformly continuously on the pairs $(x,y)$ on the pairs $(x,y)\in \Omega$. Since $\Omega$ is dense in $\mathcal{E}$, the measurable function $H^{c}$ on $\Omega$ therefore admits a continuous extension to $\mathcal{E}$. Now it is easy to see that for any $(x,y)\in \mathcal{E}$, $h^c_{x,y}$ is linear on the charts $\{\Phi_{x}\}_{x \in M}$ with derivative $H^c_{x,y}$ (cf. \cite{BX}). As a corollary, we proved Proposition \ref{prop: non-pathologic imply Wc good in Wic} which implies Theorem \ref{main} in the case $\dim E^i=1$ for some $i$.

\end{proof}



\section{Proof of Theorem \ref{main: K-S conservative}}\label{section: conservative KS}Recall that in Theorem \ref{main: K-S conservative}, $\al:\ZZ^k\to \Diff^\infty_{\mathrm{vol}}(M)$ is a higher rank smooth volume preserving Anosov action satisfying the same condition as assumption (A) of Theorem \ref{main}. Firstly we have the following higher-rank properties for $\al$.

\begin{lemma}\label{lemma: assp A+higher rank imply tns} 
Let $\al:\ZZ^k\to \Diff^\infty_{\mathrm{vol}}(M)$ be not rank-one smooth volume preserving Anosov action which satisfies the following: for every coarse Lyapunov distribution $E_i$, there exists a regular element $a\in \ZZ^k$ such that $E^i=E^u_a$.  Then the following hold: 
\begin{enumerate} \item $\al$ is totally non-symplectic, i.e. there is no Lyapunov functional $\chi,\chi'$ such that $\chi$ is negatively proportional to $\chi'$.
\item $\al$ is resonance-free in the sense that for any Lyapunov functionals $\chi_i,\chi_j,\chi_k$, $\chi_i-\chi_j$ is not proportional to $\chi_k$.
\end{enumerate} 
\end{lemma}
\begin{proof}
(1) Suppose that there are $i, j$ such that $\chi_i=c\chi_j$ for some $c<0$. Then these two Lyapunov functionals share the same Weyl chamber wall i.e. $\ker \chi_i=\ker \chi_j$. Since $\al$ is assumed to be not rank-one, there is at least one more Lyapunov exponent $\chi_k$ which is not proportional to   $\chi_i$ and $\chi_j$. By assumption (A) there exists a regular element $a$ such that $E^k=E^u_a$. This implies $\chi_k(a)>0$, but $\chi_i(a)<0$ and $\chi_j(a)<0$.  The last two inequalities are not possible for any regular element because  $\chi_i$ and $\chi_j$ are negatively proportional. 

(2) Suppose that there are three Lyapunov functionals $\chi_i, \chi_j$ and $\chi_k$ such that $\chi_i-\chi_j = c\chi_k$ for some $c\ne 0$. Direct consequence of the assumption on $\al$ is that there exists regular element $a$  for which  $E^j=E^s_a$. Then $\chi_j(a)<0$, but $\chi_i(a)>0$ and $\chi_k(a)>0$. This implies $c>0$.

On the other hand, there exists regular element $b$ for which $E^i=E^u_b$. Then $\chi_i(b)>0$, but $\chi_j(b)<0$ and $\chi_k(b)<0$. This implies $c<0$. Therefore we can conclude $c=0$ which contradicts the assumption.
\end{proof} 

Firstly we proved Theorem \ref{main: K-S conservative} under the assumptions of whether there exists one-dimensional $E^i$.
\subsection{Case 1: there exists $i$ such that $\dim E^i=1$}\label{subsection: case 1 chapter proof KS} Since there exists $i$ such that $E^i$ is one dimensional, by our assumption of $\al$ (assumption A in Theorem \ref{main}) there exists $a_i\in A$ such that $\al(a_i)$ is a codimensional one Anosov diffeomorphism on $M$. By Franks-Newhouse theorem (cf. \cite{Franks},\cite{Newhouse}), $M$ is topological a torus. By a classical result of Franks and Manning we could find a homeomorphism $h:M\to M$ such that $h\circ \al(a_i)\circ h^{-1} $ is a linear automorphism on the torus. By commutivity, we can conclude that $h\circ \al\circ h^{-1}$ is a linear action on the torus. 

Notice that by Lemma \ref{lemma: assp A+higher rank imply tns}, our action $\alpha$ is totally non-symplectic. Then the linear action $h\circ \al\circ h^{-1}$ is also totally non-symplectic by the remark (5) in \cite{TNK}. It is easy to see that any totally non-symplectic linear action on the torus, in particular $h\circ \al\circ h^{-1}$, has no rank one factor in the sense of \cite{HW} (because it would have modulo a finite group all elements with same stable and unstable and these cannot be in the same stable for some element on the action\color{black}). Since  there exists one element of $\al$ which is Anosov, then by \cite{HW} we know that $\al$ conjugates to its linearization (i.e. $h\circ \al\circ h^{-1}$) by a smooth diffeomorphism on torus.
\subsection{Case 2: every $E^i$ is higher dimensional}

 As before we denote by $$W^{\hat{i}}:=W^{1,\dots,i-1,i+1,\dots}$$ Moreover we denote the holonomy along $W^{\hat{i}}$ foliation between local $W^i$ leaves by $h^{\hat{i}}$.
\subsubsection{Smoothness of coarse Lyapunov distributions}
\begin{lemma}\label{lemma: smooth foliation thm KS}
All coarse Lyapunov foliations $W^i$  are smooth foliations of $M$.
\end{lemma}
\begin{proof}
By our assumption for $\al$, as the proof of Proposition \ref{coro: all fol integrable},  we have that for any $W^i$, $W^i$ has uniformly smooth leaves and for any non-empty set of indices $J$, $\oplus_{j\in J}E^j$ is integrable.

Then by Theorem 1.3, 1.4 of \cite{Sadovskaya} and our assumption of $A$, $h^{\hat{i}}$ is uniformly smooth. Therefore $W^{\hat{i}}$ is a uniformly smooth foliation for any $i$. Apply Lemma \ref{lemma: intersection regular} several times we know that each $W^i$ is a smooth foliation of $M$.
\end{proof}

\subsubsection{Linearization, connection and Holonomy}\label{subsection linearization}
For each Lyapunov distribution $E^i$, we pick the assoicatied $a_i\in A$ in the assumption A of $\al$.
By Lemma \ref{lemma Sadov linearization}, there is a  unique family of coordinates $\{\Phi^i_x: E^i_x\to W^i(x)\}_{x\in M}$ such that 
\begin{enumerate}
\item $\Phi^i_{\al(a_i)(x)} \circ D\al(a_i)_{x} = \al(a_i) \circ \Phi^i_{x}$,
			
			\item $\Phi^i_{x}(0) = x$ and $D_{0}\Phi^i_{x}$ is the identity map,
			
			\item The family $\{\Phi^i_{x}\}_{x \in M}$ varies continuously with $x$ in the $C^{\infty}$ topology. 
\end{enumerate}

As in Step 3 of the proof of Proposition \ref{prop: single diff, non-pathologic imply smoothness},  we can consider the unstable Holonomy and stable Holonomy for cocycle $D\al(a_i)|_{E^i}$, we denote them by $H^i$ and $H^{\hat{i}}$ respectively: for any $x,y\in M$ we let $I^i_{xy}:E^i(x)\to E^i(y)$ to be a linear identification which is $C^1$ close to the identity (since $E^i$ is smooth), then for any $y\in W^i_{loc}(x)$,$$H^i_{xy}:=\lim_{n\to\infty}D\al(a_i)^n_{\al(a_i)^{-n}y}\circ I^i_{\al(a_i)^{-n}x~\al(a_i)^{-n}y}\circ D\al(a_i)^{-n}_x|_{E^i}$$
Moreover $H^i_{xy}$ defined above is the unique collection of linear identification from $E^i(x)$ to $E^i(y)$ satisfying the equations in Lemma \ref{lemma: Holonomies} i.e. for any $y,z\in W^i_{loc}(x)$,
\begin{itemize}
\item[a.] $H^{i}_{xx} = Id$ and $H^{i}_{yz} \circ H^{i}_{xy} = H^{i}_{xz}$;
\item[b.] $H^{i}_{xy} = D\al(a_i)^{n}_{\al(a_i)^{-n}y} \circ H^{i}_{\al(a_i)^{-n}x~\al(a_i)^{-n}y} \circ D\al(a_i)^{-n}_{x}$  for any $n \geq 0$. 
\item[c.] $\| H^{i}_{xy} - I^i_{xy}\| \leq Cd(x,y)$ 
\end{itemize}

Similar for any $y\in W^{\hat{i}}_{loc}(x)$, we define $$H^{\hat{i}}_{xy}:=\lim_{n\to\infty}D\al(a_i)^{-n}_{\al(a_i)^{n}y}\circ I^i_{\al(a_i)^{n}x~\al(a_i)^{n}y}\circ D\al(a_i)^{n}_x|_{E^i}$$
and $H^{\hat{i}}$ is the unique collection of linear identification satisfying that for any $y,z\in W^{\hat{i}}_{loc}(x)$,

\begin{itemize}
\item[a'.] $H^{\hat{i}}_{xx} = Id$ and $H^{\hat{i}}_{yz} \circ H^{\hat{i}}_{xy} = H^{\hat{i}}_{xz}$;
\item[b'.] $H^{\hat{i}}_{xy} = D\al(a_i)^{-n}_{\al(a_i)^{n}y} \circ H^{\hat{i}}_{\al(a_i)^{n}x~\al(a_i)^{n}y} \circ D\al(a_i)^{n}_{x}$  for any $n \geq 0$. 
\item[c'.] $\| H^{\hat{i}}_{xy} - I^i_{xy}\| \leq Cd(x,y)$ 
\end{itemize}

By Lemma \ref{lemma: affine transitions}, as in \cite{Fang}, by $\Phi^i_x, \Phi^i_y$ we pull back the canonical flat linear connections
of $E^i(x), E^i(y)$ respectively onto $W^i(x)$, we get the same smooth connection on $W^i(x)$. Therefore we get a well-defined transversely continuous $\al(a_i)-$invariant connection along $W^i$, we denoted it by $\nabla^i$. 

We give the geometric explanations for $H^i$ and $H^{\hat{i}}$. It is easy to see that the parallel transport induced by $\nabla^i$ along $W^i$ is $H^i$ (in fact by flatness and $\al(a_i)-$invariance of $\nabla^i$, we have that the parallel transport induced by $\nabla^i$ does not depend on the choice of path and satisfying equations a., b. and c. ).

For $H^{\hat{i}}$, we have the following lemma:
\begin{lemma}\label{lemma: smooththness of H hat i}
For any $y\in W^{\hat{i}}_{loc}(x)$, $Dh^{\hat{i}}_{xy}=H^{\hat{i}}_{xy}$. In particular, $H^{\hat{i}}_{xy}$ depends smoothly on $x,y$.
\end{lemma}
\begin{proof}Since all coarse Lyapunov distributions of $\al$ are smooth, $Dh^{\hat{i}}$ is well-defined and satisfying the equation c' above. And the equations a'. and b'. hold for $Dh^{\hat{i}}$ since $W^i, W^{\hat{i}}$ are $\al(a_i)-$invariant foliation. So by uniqueness of $H^{\hat{i}}$ we have $H^{\hat{i}}=Dh^{\hat{i}}$. The smoothness of $H^{\hat{i}}$ is a consequence of smoothness of $W^i$ and $W^{\hat{i}}$.
\end{proof}
\subsubsection{Conformal structure and a metric on $E^i$.}
Since $\al(a_i)$ is uniformly quasiconformal on $E^i$, and $a_i$ is volume preserving, by Proposition 22. of \cite{BX} (or \cite{Fang}, \cite{Sadovskaya}) there is a continuous conformal structure $\tau^i$ on $E^i$ which is invariant under $H^i$ and $H^{\hat{i}}$ (redefine $\tau^i$ on a neglible set if necessary). As a result, $\tau^i$ is uniformly smooth along $W^i$ and $W^{\hat{i}}$. By Journ\'e lemma \cite{Journe}, $\tau^i$ is a smooth conformal structure on $E^i$. 

For $E^i$, we consider the associated Lyapunov functional $\chi_i$, and denote by $\dim E^i:=d_i$. Our goal is the following crucial lemma:
\begin{lemma}\label{lemma: exist good metric}For any $i$, there exists a smooth metric $g^i$ on $E^i$ such that for any $v\in E^i$, 
\begin{equation}\label{equation: ai uniform expanding Ei}
\|D\al(a_i)\cdot v\|_{g^i}=e^{\chi_i(a_i)}\|v\|_{g^i}
\end{equation}
\end{lemma}
\begin{proof}The proof is similar to that in Chapter 3.2 in \cite{KS07}. Consider the $\RR^k$ action $\tilde{\al}$ on a manifold $\tilde{M}$ defined by the standard suspension of $\al$, cf. \cite{KS07}. Since all nontrivial elements of $\al$ are Anosov, by Lemma \ref{lemma: assp A+higher rank imply tns} $\tilde{\al}$ satisfying all assumptions of Proposition 2.3. in \cite{KS07}. We denote the lift of $E^i$ by $\tilde{E}^i$. 
As in \cite{KS07}, we fix a background metric $\tilde{g}^i_0$ on $\tilde{E}^i$. Then $\tilde{g}^i_0$ induced a volume form $\tilde{\nu}^i_0$ on $\tilde{E}^i$, we defined the function $q$ as the following: for any $x\in \tilde{M}, b\in \RR^k$,
\begin{equation}q(x,b):=\mathrm{Jac}_{\tilde{\nu}^i_0}(D\tilde{\al}(b)_x|_{\tilde{E}^i})
\end{equation}

Since $\al$ is totally non-symplectic, we can choose a \textit{generic singular} element (cf.\cite{KS07}) $a\in \RR^k$ in the Lyapunov hyperplan associated to functional $\chi_i$ such that  $\tilde{\al}(ta), t\in \RR$ acts transitively on $\tilde{M}$. Then by the same proof as in Chapter 3.2 of \cite{KS07} (that is the only place we need the condition all nontrivial elements of $\al$ are Anosov), there is a continuous positive function $\tilde{\phi}^i$ on $\tilde{M}$ such that for any $x\in \tilde{M}$
\begin{equation}\label{equation: coho eqn gene sing a}
\tilde{\phi}^i(x)\cdot\tilde{\phi}^i(\tilde{\al}(a)\cdot x)^{-1}=q(x,a)
\end{equation}
We claim that for any $x\in \tilde{M}, b\in \RR^k$,
\begin{equation}\label{equation: coho eqn all b}
\tilde{\phi}^i(x)\cdot\tilde{\phi}^i(\tilde{\al}(b)\cdot x)^{-1}=e^{-d_i\chi_i(b)}q(x,b)
\end{equation}
In fact, we can define a continuous volume form on $\tilde{E}^i$ by $\tilde{\nu}^i:=\tilde{\phi}^i\cdot \tilde{\nu}^i_0$. By \eqref{equation: coho eqn gene sing a}, $\tilde{\nu}^i$ is invariant under the action of $\tilde{\al}(ta), t\in \RR$. By commutativity, $\tilde{\al}(b)_\ast \tilde{\nu}^i=\psi\cdot \tilde{\nu}^i$ is also an $\tilde{\al}(ta)-$invariant volume form on $\tilde{E}^i$. Since $\al(ta)$ acts transitively on $\tilde{M}$, $\psi$ is a constant function. Notice that $b$ acts on $\tilde{E}^i$ with Lyapunov exponents $\chi_i(b)$, by constancy of $\psi$ we get \eqref{equation: coho eqn all b}.

In particular, by restriction on $M$ (the zero section of the fiberation $\tilde{M}$ is $M$),  there is a continuous function $\phi^i$ on $M$ such that for any $x\in M$,
\begin{equation}\label{equation: coho eqn for ai on M}
\phi^i(x)\cdot \phi^i(\al(a_i))^{-1}=e^{-d_i\chi_i(a_i)}q(x,a_i)
\end{equation}
But $\al(a_i)$ is a volume preserving Anosov diffeomorphism on $M$ (hence transitive) and the right hand side of \eqref{equation: coho eqn for ai on M} is a smooth function (since $E^i$ is smooth in $M$), then by cocycle regularity theorem in \cite{Llave}, $\phi^i$ can be taken to be a smooth function.

As a result, $\phi^i$ induces a smooth volume form $\nu^i$ for $E^i$ on $M$ such that $$\al(a_i)_\ast \nu^i=e^{d_i\chi_i(a_i)}\nu_i$$
combine with the $\al(a_i)-$invariant conformal structure $\tau^i$ on $E^i$, we easily get a smooth metric $g^i$ on $E^i$ satisfying \eqref{equation: ai uniform expanding Ei}.
\end{proof}

\subsubsection{Smoothness and invariance of $\nabla^i$.}
Consider the leafwise Levi-Civita connection of the smooth metric $g^i$, we denoted it by $\bar{\nabla}^i$. Obviously $\bar{\nabla}^i$ is smooth and by \eqref{equation: ai uniform expanding Ei}, $\bar{\nabla}^i$ is $\al(a_i)-$invariant. As in the proof of Lemma 3.1.2 of \cite{Fang}, the family of exponential maps $$\bar{\Phi}^i_x:=\exp^{\bar{\nabla}^i}(x): E^i(x)\to W^i(x), x\in M$$
is well-defined, smoothly depend on $x$ and satisfying  the assumptions (1),(2),(3) in section \ref{subsection linearization} for coordinates $\{\Phi^i_x\}_{x\in M}$. Therefore by uniqueness of coordinates $\Phi^i$, we have that $\bar{\Phi}^i=\Phi^i$. As a consequence, $\Phi^i_x$ also smoothly depend on $x$. Therefore $\nabla^i$ is a smooth connection for the bundle $E^i$ (A priori $\nabla^i$ is smooth along $W^i$ and only continuous transversely). 

By commutativity and conditions a. b. and c. for $H^i$ in section \ref{subsection linearization} it is easy to see that for any $y\in W^i(x)$, $a\in A$ the following diagram commutes:
$$\xymatrix{E^i(x)\ar[rr]^{D\al(a)}\ar[d]^{H^i_{xy}}& &E^i(\al(a)x)\ar[d]^{H^i_{\al(a)x~\al(a)y}}\\
	E^i(y)\ar[rr]^{D\al(a)}& & E^i(\al(a)y)} $$
Therefore by the geometric explanation for $H^i$, the parallel transport induced by $\nabla^i$ is $\al-$invariant. As a corollary, $\nabla^i$ is $\al-$invariant as well (since the connection can be recovered by its parallel transport).
\subsubsection{Proof of case 2}
As in \cite{KS07}, \cite{Fang} we construct a smooth $\al-$invariant connection $\nabla$ as the following:
Let $X,Y$ be two vector fields on $M$. We decompose $X,Y$ as $\sum X^i$ and $\sum Y^i$ under the coarse Lyapunov splitting $\oplus E^i$. Then $$\nabla_XY:=\sum_i \nabla^i_{X^i}Y^i+ \sum_{i\neq j} \Pi_j [X^i, Y^j] $$
where $\Pi_j$ is the projection onto $E^j$ with respect to the coarse Lyapunov splitting. Since $\nabla^i, E^i$ are $\al-$invariant and smooth, so is $\nabla$.

We take an arbitrary non-trivial element $\al(a)$, notice that $\al(a)$ is a transitive Anosov diffeomorphism, preserving a smooth connection and has smooth stable and unstable foliation. By \cite{BL}, $\al(a)$ is conjugated to an affine automorphism on an infranilmanifold $N$ by a smooth diffeomorphism $h$. By commutativity, $\al$ is also conjugated to an affine abelian action on an infranilmanifold by the same conjugation $h$ since any diffeomorphism commuting with an Anosov automorphism of an infranilmanifold is an affine automorphism itself, cf.\cite{Hurder, PaYo}. By Lemma \ref{lemma: assp A+higher rank imply tns}, $\al$ is resonance free. Then as in section 3.5 of \cite{KS07}, $N$ is finitely covered by a torus.

\section{Proof of Theorems \ref{main: intro}, \ref{main: global rigidity}, and \ref{main: gl r dim 1}}\label{section: proof thm 134}
\subsection{Proof of Theorem \ref{main: global rigidity}}
Recall that by assumptions of Theorem \ref{main: global rigidity}
the action $\al$ satisfies all assumptions in Theorem \ref{main}, and $W^c$ is not pathological. Then by Theorem \ref{main}, $W^c$ is a smooth foliation in $M$. $N:=M/ W^c$ is a smooth manifold, the quotient map $\pi: M\to N$ is a smooth submersion and $\al$ induces a smooth Anosov action $\bar{\al}: A\to \Diff^\infty(N)$ such that $\bar \al\circ \pi= \pi \circ \al$. Moreover, $\nu$ induces a $\bar\al-$invariant volume form $\nu_N$ on $N$ (for any measurable set $K\subset N$, $\nu_N(K):=\nu_M(\pi^{-1}(K))$, by smoothness of $W^c$, $\nu_N$ is a volume form on $N$). So $\bar\al$ is also a volume preserving action. Since all non-trivial elements in $A$ for $\al$ are regular which implies all nontrivial elements of $\bar{\al}$ are Anosov, therefore $\bar{\al}$ satisfying all conditions of Theorem \ref{main: K-S conservative}. Therefore after a finite cover $\bar{\alpha} $ is smoothly conjugated to a $\ZZ^k-$action formed by affine automorphisms of a torus. 

We need the following purely higher rank result to prove that $\al$ is essentially a product over Anosov linear action. A similar cocycle rigidity result is proved recently in \cite{DX}. Recall that since $\al$ satisfies all assumptions in Theorem \ref{main}, as in the discussion in Chapter \ref{section: proof of thm 1, dim>1}, we have the coarse Lyapunov splitting for $\al$: $TM=E^1\oplus \cdots E^n\oplus E^c$.
\begin{prop}\label{prop: joint integ}The distribution $E^i$ is $C^\infty$ for any $i$. Moreover $\oplus_{i=1}^n E^i$ is integrable.
\end{prop}
\begin{proof}The proof is basically the same as the proof of Proposition 5.1 in \cite{DX}. For completeness we give an outline here. 

Recall that by condition A, each $E^i$ is the unstable distribution of a smooth partially hyperbolic diffeomorphism, so $W^i$ has uniformly smooth leaves. Moreover by the same proof as in Lemma \ref{lemma: assp A+higher rank imply tns}, $\bar\al$ is totally non-symplectic and resonance free. Then by the same proof of Proposition 5.7 of \cite{DX} (see also \cite{KS07}) we can prove that $E_i$ is uniformly $C^{\infty}$ along $W^j$ for any $i,j$. 

On the other hand, notice that by our assumption (B'), $\al$ is $\infty-$bunched, since $W^c$ has uniformly smooth leaves, by the same proof as in Proposition 5.8 of \cite{DX}, we have that $E^i$ is uniformly $C^\infty$ along $W^c$ (cf. \cite{Pugh-Shub-Wilkinson}). 

Recall that in Corollary \ref{coro: all fol integrable} we proved that for any non-empty subset $J\subset \{1,\dots,n,c\}$, $J\neq\{1,\dots, n\}$, $\oplus_{j\in J}E^j$ is integrable. Then as in the discussion around Lemma 5.9 in \cite{DX} we apply Journ\'e Lemma \cite{Journe} inductively, we get $E^i$ is actually a smooth distribution and $W^i$ is a smooth foliation.

Now we prove integrablity of $\oplus_{i=1}^n E^i$. By Corollary \ref{coro: all fol integrable}, $E^j\oplus E^k$ is integrable. Then for any $C^1-$vector fields $X=\sum X^i,Y=\sum Y^i$ contained in $\oplus_{i=1}^n E^i$, we have $$[X,Y]=\sum [X^i, Y^i]\oplus \sum_{j\neq k}[X^j, Y^k]$$
Notice that $[X^i,Y^i]$ is contained in $E^i$, $[X^j, Y^k]$ is contained in $E^j \oplus E^k$, then by Frobenius theorem, $\oplus_{i=1}^n E^i$ is integrable.
\end{proof}

We prove Theorem \ref{main: global rigidity} now. Basically it is a corollary of foliation theory and Proposition \ref{prop: joint integ}. By Proposition \ref{prop: joint integ} $\oplus E_i$ is tangent to a $\al-$invariant \textit{horizontal} foliation $\mathcal{W}_H$ of the fiber bundle $(\pi, M, N=M/W^c)$ in the sense that $\mathcal{W}_H$ is uniformly transverse to the fibers of $W^c$ in $M$.

We take an arbitrary point $x_0\in M$. Consider the universal covering space $(p, \hat{M}, \hat{x_0})$ of $(M, x_0)$. Consider the universal cover $(p, \hat{N}, \hat{x_0})$ of $(N, x_0)$ where $p$ is the covering map such that $p(\hat{x_0})=x_0$. Then $p$ induces a pullback bundle (cf.\cite{St}) $\hat{\pi}:\hat{M}\to \hat{N}$ with the same fiber $W^c$ as $\pi: M\to N$, moreover there is a covering map $\hat{p}:\hat{M}\to M$ which preserves the fiber such that the following diagram commutes: 
$$\xymatrix{\hat{M}\ar[rr]^{\hat{p}}\ar[d]^{\hat{\pi}}& &M\ar[d]^{\pi}\\
	\hat{N}\ar[rr]^{p}& & N}$$

We denote by $\hat{\al}$ the liftings of $\al$ on $\hat{M}$,  $\hat{\bar\al}$ the lifting of $\bar\al$ on $\hat{N}$. There is a horizontal foliation $\hat{\mathcal{W}}_H$ of $\hat{M}$, which is lifted from $\mathcal{W}_H$ by $\hat{p}$. $\hat{\mathcal{W}}_H$ is invariant under the lifted action $\hat{\al}$ on $\hat{M}$. Moreover by theory of suspension in foliation theory (cf. pp. 124, section 1.2 of \cite{HH} or \cite{NTnonabelian}) $\hat{\mathcal{W}}_H$ is a global section of the fiber bundle $\hat{M}\to \hat{N}$ in the sense that each leaf of $\hat{\mathcal{W}}_H$ intersects each fiber $W^c$ at exactly one point.

As a result, we can define a uniformly $C^\infty$ diffeomorphism $\varphi$ which is induced by the holonomy of $\hat{\mathcal{W}}_H$ in $\hat{M}$: 
$$\varphi: \hat{N}\times W^c(\hat{x_0})\to \hat{M}, ~~ \varphi(\hat{x},y):=\hat{\mathcal{W}}_H(y)\cap W^c({\hat{x}})$$
Since $\hat{\mathcal{W}}_H$ is a global section, $\varphi$ is well-defined. Moreover by $\hat{\al}-$invariance of $\hat{\mathcal{W}}_H$, for any $a\in A, \hat{x}\in \hat{N}, y\in W^c({x_0})$, we have $$\varphi^{-1}\circ \hat{\al}(a, \varphi(\hat{x},y))=(\hat{\bar{\al}}(a)\cdot \hat{x}, \hat{\mathcal{W}}_H(\hat\al(a)\cdot y)\cap W^c(\hat{x_0}))$$
Notice that the second coordinate on the right hand side does not depend on $\hat{x}\in \hat{N}$ which means $\varphi^{-1}\circ \hat{\al}\circ \varphi$ is a constant cocycle taking values in $\Diff^\infty(W^c(\hat{x_0}))$ over $\hat{\bar\al}$. By further conjugation of $\hat{\bar\al}$ (by Theorem \ref{main: K-S conservative}) we know $\hat{\al}$ is smoothly conjugated to a product over a $\ZZ^k-$action which is a lifting of an algebraic Anosov action on a torus.


\subsection{Proof of Theorem \ref{main: gl r dim 1}}
\label{subsection: proof Thm circle ext}
Suppose in addition $\dim E^c=1$. Then by Theorem \ref{main}, $\al$ is an isometric extension over $\bar\al$. There is a smoothly varing family of Riemannian metrics $\{d_x\}$ on each center leaf and $\al(a)$ is an isometry on each center leaf with respect to the metric $\{d_x\}$ for any $a\in \ZZ^k$.

Therefore the total arclength of   the center leaf is also $\al-$invariant. As a corollary, by considering a transitive element $\bar\al(a)$, each center leaf has the same total arglength. Without loss of generality by further rescaling we assume each center leaf has total arclength $1$ with respect to the metrics $\{d_x\}$.

Pass to a double cover if necessary we can fix a global defined continuous varying orientation $\tau_x$ on each leaf of $W^c(x)$. Combine with the metric $\{d_x\}$ on center foliation, there is a smooth $\mathbb{T}^1-$action on $M$: for any $\theta\in \mathbb{T}^1, x\in M$, , $R_\theta\cdot x$ is defined by the unique point $y\in W^c(x)$ such that the oriented arc $\widehat{xy}$ has length $\theta$. Moreover $R_\theta$ preserves the fibers of fiber bundle $\pi: M\to M/W^c$ and acts freely and transitively on each fiber. Therefore $\pi: M\to M/W^c$ is a principal $\mathbb{T}^1-$bundle. 

By Lemma \ref{lemma: assp A+higher rank imply tns} and Theorem \ref{main: K-S conservative}, up to a smooth conjugacy we may assume that  $\bar\al$ is a totally non-symplectic linear Anosov action on an infranilmanifold (which is finitely covered by a torus). As a result, up to a double cover and a smooth conjugacy, $\al$ is a principal $\TT^1-$extension over a totally non-symplectic linear Anosov action on an infranilmanifold. Then by Theorem 7.1 in \cite{NTnonabelian}, $\al$ is essentially algebraic in the sense of section \ref{subsection main examp}.

\subsection{Proof of Proposition \ref{prop: restate main intro} and Theorem \ref{main: intro}}
Recall that in Proposition \ref{prop: restate main intro} and Theorem \ref{main: intro} $\al\in PH^\infty_{\mathrm{vol}}(k,M)$ is a maximal Cartan partially hyperbolic action and the common center foliation $W^c$ for $\al$ is uniformly compact and one-dimensional. The following lemma allows us to apply Theorem \ref{main} to prove Proposition \ref{prop: restate main intro}.
\begin{lemma}\label{lemma: maximal imply cond A}$\al$ is full.
\end{lemma}
\begin{proof}(cf. \cite{DX} also) Since $\al$ is maximal, then it has exactly $k+1$ Lyapunov hyperplanes in general position. This implies that obviously there must be at least two Lyapunov hyperspaces, and  that there are exactly $2^{k+1}-2$ Weyl chambers. Since there is no Weyl chamber where all Lyapunov exponents are positive (or all negative), it follows that all combinations of signs appear among Weyl chambers, so for any $\chi_i$ there is a  Weyl chamber $\mathcal{C}$ in which $\chi_i$ is positive while all other non-positively proportional Lyapunov functionals are negative. 

By Definition \ref{def: PHmathrmvolinfty(k, M)} we know  there exists a partially element $b_{\mathcal{C}}$ in $ \mathcal{C}$.  We claim that 
\begin{eqnarray}\label{eqn:Eu=Echi_i}
E^u_{b_{\mathcal{C}}}&=&E_{\chi_i}, \nu_M-a.e. \\\label{eqn: Es=E rest}
E^s_{b_{\mathcal{C}}}&=&\bigoplus_{\lambda\neq c\chi_i, c>0}E_\lambda, \nu_M-a.e. 
\end{eqnarray}
which implies the action is full (see section \ref{subsection: def PH} for precise definition of fullness).

In fact by definition of Lyapunov functional and Lyapunov subspaces in section \ref{subsection: def PH} we know $\bigoplus_{\chi_i}E_{\chi_i}=E_T=E^s_{b_{\mathcal{C}}}\oplus E^u_{b_{\mathcal{C}}}$ Therefore for $\nu_M$ almost every  $x$, 
\begin{eqnarray*}
E_{\chi_i}(x)-\{0\}&=&\{v\in E_T(x), \lim_{n\to-\infty}\log\|D\al(nb_\mathcal{C})\cdot v\|<0 \}=E^u_{b_C}-\{0\}
\end{eqnarray*}
Similarly we could prove that for $\nu_M-$almost every $x$, $$\bigoplus_{\lambda\neq c\chi_i, c>0}E_\lambda(x)-\{0\}=E^s_{b_C}-\{0\}$$ 
Therefore \eqref{eqn:Eu=Echi_i}, \eqref{eqn: Es=E rest} hold.
\end{proof}

Since $\al$ is full, there exists $a\in \ZZ^k$ such that $\al(a)$ is a partially hyperbolic diffeomorphism with central codimension $1+k$ type (cf. \cite{BoThesis}) then by Theorem C of \cite{BoThesis} we know up to a double cover all leaves of the central foliation $W^c$ for $\al$  have trivial holonomy. Therefore $\al$ satisfy assumptions  (A), (B'), (C) in section \ref{subsection: def ABC}. 

If $W^c$ is pathological then (1). of Proposition \ref{prop: restate main intro} holds. If $W^c$ is not pathological, by Theorem \ref{main} we know $W^c$ is an isometric extension over an Anosov action $\bar\al\in PH^\infty_{\mathrm{vol}}(k, N)$ where $N=M/W^c$. Moreover $\bar\al$ also satisfies condition (A). Therefore by Lemma \ref{lemma: assp A+higher rank imply tns} and discussion in section \ref{subsection: case 1 chapter proof KS} we know $\bar\al$ is smoothly conjugated to a linear Anosov action on a torus. And $\al$ is an isometric circle extension over $\bar\al$. But in section \ref{subsection: proof Thm circle ext} we actually proved that any isometric circle extension over linear TNS Anosov action on a torus is essentially algebraic. Then by Lemma \ref{lemma: assp A+higher rank imply tns} we know $\al$ is essentially algebraic which means (2). of Proposition \ref{prop: restate main intro} holds. Therefore we get the proof of Proposition \ref{prop: restate main intro} hence Theorem \ref{main: intro}.

\section{Proof of Theorems \ref{main for single diffeo} and  \ref{main for single diffeo'}}\label{section: proof thm rank-1}
Theorem \ref{main for single diffeo} is basically a corollary of Proposition \ref{prop: single diff, non-pathologic imply smoothness}. Recall that in Theorem \ref{main for single diffeo} $f:M\to M$ is a smooth volume preserving partially hyperbolic diffeomorphism with compact center foliation, moreover $\dim E^u=\dim E^s=1$, then by Theorem A of \cite{BoThesis}, we know that up to a double cover of $f$, $W^c(f)$ has trivial holonomy. 

Suppose $W^c(f)$ is not pathological in the sense that $\nu_M$ has Lebesgue disintegration along $W^c$ and $W^c$ is strongly absolutely continuous in both of $W^{cs}, W^{cu}$, then the center holonomy $h^{c,s}, h^{c,u}$ of $f$ within $W^{cs}$ and $W^{cu}$ respectively are both absolutely continuous. Therefore up to a double cover, $f, f^{-1}$ satisfy all assumptions of Proposition \ref{prop: single diff, non-pathologic imply smoothness}. 

As a result, by Proposition \ref{prop: single diff, non-pathologic imply smoothness} $h^{c,u}$ (or $h^{c,s}$) is uniformly smooth within $W^{cu}$ (or $W^{cs}$) restricted on $W^u_{loc}$ (or $W^s_{loc}$). If $f$ is $r-$bunching, then $W^{cu},W^{cs},W^c$ have $C^{r+}-$leaves. Since $W^{u(s)}$ have smooth leaves, by uniform smoothness of $h^{c,u}$ (or $h^{c,s}$), applying Lemma \ref{lemma: reg of holonomy and foliations},  we get that $W^c$ is a $C^{r+}-$foliation within $W^{cu}$ and $W^{cs}$. Then by Lemma \ref{lemma: Foliation version Journe Lemma} we know that $W^c$ is a $C^{r+}$ foliation of manifold $M$. Then the rest of proof for Theorem \ref{main for single diffeo} is the same as the proof of Theorem \ref{main} (see discussion after Proposition \ref{prop: smoothness of foliations}).

Similarly we can prove Theorem \ref{main for single diffeo'}, by the same argument we can prove that for diffeomorphism $f$ in Theorem \ref{main for single diffeo'}, if $W^c(f)$ is not pathological, then we have $W^c(f)$ is a $C^{r+}$ foliation in $W^{cs}(f)$ (by Proposition \ref{prop: ac implies linearity of center fol}) and $W^{cu}(f)$ (by Proposition \ref{prop: single diff, non-pathologic imply smoothness}) respectively, then the rest of proof for Theorem \ref{main for single diffeo'} is the same as the proof of Theorem \ref{main for single diffeo}.

\end{document}

%% file: makros.tex
\newtheorem{maintheorem}{Theorem} 
\newtheorem{theorem}{Theorem}
\newtheorem{lemma}[theorem]{Lemma}
\newtheorem{coro}[theorem]{Corollary}
\newtheorem{prop}[theorem]{Proposition}

\newtheorem*{assumptions*}{Assumptions}

\newtheorem*{rem*}{Remark}

\theoremstyle{remark}
\newtheorem{remark}[theorem]{Remark}
\newtheorem*{remark*}{Remark}

\theoremstyle{definition}
\newtheorem{definition}{Definition}

\newcommand{\RR}{{\mathbb R}}

\newcommand{\TT}{{\mathbb T}}

\newcommand{\ZZ}{{\mathbb Z}}

\newcommand{\be}[1]{\begin{equation} \label{#1} }
\newcommand{\ee}{\end{equation}}
\newcommand{\beq}{\begin{equation}}